\documentclass[12pt]{amsart}
\usepackage[margin=1in]{geometry}
\usepackage[latin1]{inputenc}
\usepackage[dvipsnames]{xcolor}
\usepackage{fullpage}
\usepackage{array}
\usepackage{enumerate}
\usepackage{amsthm,amsmath, amssymb, amsfonts,multicol}
\usepackage{verbatim}
\usepackage{tikz}
\usetikzlibrary{fit,matrix,positioning,shapes}
\usetikzlibrary{decorations.shapes, patterns}
\usepackage{graphicx,float}
\usepackage{mathtools,hyperref}
\usepackage{ytableau}
\usepackage{thmtools}
\usepackage{thm-restate}

\usepackage{stmaryrd}
\usepackage{latexsym} 
\usepackage{subfigure}
\usepackage[all]{xy}
\usepackage{graphics}
\usepackage{pdfsync}
\usepackage{bm}
\usetikzlibrary{arrows.meta}
\usetikzlibrary{shapes.geometric}

\usetikzlibrary{decorations.pathreplacing}

\theoremstyle{definition}

\newtheorem{theorem}{Theorem}[section]
\newtheorem{lemma}[theorem]{Lemma}
\newtheorem{corollary}[theorem]{Corollary}

\newtheorem{definition}[theorem]{Definition}

\newtheorem{prop}[theorem]{Proposition}
\newtheorem*{theorem*}{Theorem}
\newtheorem{remark}[theorem]{Remark}

\newcommand{\dg}{\text{dg}}
\newcommand{\leg}{\text{leg}}
\newcommand{\arm}{\text{arm}}
\newcommand{\armL}{\text{arm}^{\text{left}}}
\newcommand{\armR}{\text{arm}^{\text{right}}}
\newcommand{\wdg}{\widehat{\text{dg}}}

\newcommand{\wt}{\text{wt}}
\newcommand{\wtx}{\text{wt}_x}
\newcommand{\wtqt}{\text{wt}_{q,t}}
\newcommand{\ws}{\widehat{\sigma}}

\newcommand{\Des}{\text{Des}}
\newcommand{\maj}{\text{maj}}
\newcommand{\coinv}{\text{coinv}}

\newcommand{\MLQ}{\text{MLQ} }
\newcommand{\HHL}{\text{HHL} }
\newcommand{\NAF}{\text{NAF}}
\newcommand{\MF}{\text{MF}}

\def\T#1#2{\mathcal{T}_{\left(#1,#2\right)}}

\usepackage{xargs}

\tikzset{
    blue box/.style={
        minimum width=0.7cm, 
        minimum height=0.5cm, 
        inner sep=0pt, 
        pattern=north east lines, 
        pattern color=blue!30
    },
    red box/.style={
        minimum width=0.7cm, 
        minimum height=0.5cm, 
        inner sep=0pt, 
        pattern=north west lines, 
        pattern color=red!30
    },
    green box/.style={
        minimum width=0.7cm, 
        minimum height=0.5cm, 
        inner sep=0pt, 
        pattern=crosshatch, 
        pattern color=green!30
    }
}

\title{Compression Formula for Nonattacking Fillings of Non-symmetric Macdonald Polynomials}

\author[Stewart]{Alexander P. Stewart}
\address[A. P. Stewart]{Department of Mathematics, North Carolina State University, Raleigh, NC, USA}
\email{apstewa2@ncsu.edu}
\urladdr{\href{https://sites.google.com/ncsu.edu/alexander-stewart/home?authuser=0}{https://sites.google.com/ncsu.edu/alexander-stewart}}

\begin{document}

\maketitle
\date{}

\begin{abstract}
    Macdonald polynomials are an important family of orthogonal polynomials that have applications to statistical physics. Corteel, Mandelshtam, and Williams proved a combinatorial model involving multiline queues for partitions; however, their work does not apply to weak compositions in general. We conjecture a first step towards this direction via a variant of the nonattacking filling model for nonsymmetric Macdonald Polynomials $E_\mu$. This new model compresses the classical nonattacking filling model of Haglund, Haiman, and Loehr and mirrors the concept of trivial pairs in the multiline queue model. We prove it for a large family of weak compositions and explain the difficulty of extending this proof to the general case.
\end{abstract}

\smallskip 

\smallskip 


\section{Introduction}

Nonsymmetric Macdonald polynomials $E_\mu(x_1, \ldots, x_n;q,t) = E_\mu$ are an important family of polynomials indexed by compositions, whose coefficients are rational functions in parameters $q$ and $t$. They were introduced by Opdam~\cite{O95}, Macdonald~\cite{M03}, and Cherednik~\cite{C95} to settle conjectures related to symmetric Macdonald polynomials. Since then, a lot of research has been done on nonsymmetric Macdonald polynomials, connecting them to physics~\cite{CDL17}, DAHA~\cite{M03}, various combinatorial models \cite{AM22, CMW22, HHL08, RY11}, and other families of polynomials via specializations of $q$ and $t$: Key polynomials ($q=t=0$) and Type A Demazure atoms ($q,t \to \infty$) \cite{Ma09}. Additionally, various generalizations of $E_\mu$ include permuted basement nonsymmetric Macdonald polynomials~\cite{A19}, Koornwinder polynomials~\cite{M03}, and nonsymmetric Macdonald superpolynomials~\cite{D21}. In this paper, we are interested in studying two models: the nonattacking filling model of Haglund, Hamian, and Loehr~\cite{HHL08}, and the multiline queue model of Corteel, Mandelshtam, and Williams in~\cite{CMW22}. We write \MLQ for the latter model and \HHL for the former. 

The first combinatorial model for nonsymmetric Macdonald polynomials was due to Haglund, Haiman, and Loehr~\cite{HHL08}. Their model depends on certain fillings, called nonattacking fillings, of augmented diagrams and only works for type $\text{\emph{GL}}_{n}$, see Section~\ref{subsect:NFM} for more details. It was built on work they did for the modified Macdonald polynomials. Soon after, Ram and Yip~\cite{RY11} defined the alcove walk model which works for all Lie types. In 2018, based on the work of Cantini, de Gier, and Wheeler~\cite{CdGW15}, Corteel, Mandelshtam, and Williams~\cite{CMW22} modified the multiline queue model and showed that it is equivalent to nonsymmetric Macdonald polynomials indexed by partitions. In that same paper, their computations showed that there are fewer multiline queues than nonattacking fillings, and they noticed that a sum of a subset of permuted basement nonattacking fillings corresponded to a single multiline queue. Furthermore, they left open the question of whether there is an extension of the \MLQ model to all compositions. The primary goal of this paper is to provide a first step in that direction.

To do so, we analyze the \HHL and \MLQ models to understand the key differences between them. As a result, we introduce a variant of the \HHL model that includes the notion of a trivial pair from the \MLQ model to nonattacking fillings. The combinatorial objects in this family form a subset of nonattacking fillings that we call \emph{restricted nonattacking fillings}; see Section~\ref{section:RNF} for more details. We can view them as an extension of the \MLQ model to compositions because, in the partition case, the number of restricted nonattacking fillings is equal to the number of multiline queues; see Section~\ref{section:RNF} for a proof of this. Also, we define a modified $q,t$-weight based on the $q,t$-weight of the \HHL model that takes into account trivial pairs for general augmented diagrams. We conjecture that the restricted nonattacking fillings form a combinatorial model for $E_\mu$.

\begin{restatable}[]{conj}{MainConjecture}
\label{conj:MainConj}
Let $\mu$ be a composition and $\mathcal{R}(\mu)$ be the set
    of restricted nonattacking fillings. Then 
    $$E_\mu = \sum_{\ws \in \mathcal{R}(\mu)} \wt_{m}(\ws)x^{\sigma},$$
    where $\wt_{m}(\ws)$ is the mixed weight of the restricted 
    nonattacking filling $\ws$.
\end{restatable}

In Section~\ref{section:Compressing}, we prove that Conjecture~\ref{conj:MainConj} is true for a large family of compositions; see Section~\ref{section:Compressing} for the details. 

\begin{restatable}[]{theorem}{MainTheorem}
\label{thm:MainTheorem}
If $\mu$ is a composition such that each box $u \in \dg'(\mu)$ hits at most one box, then Conjecture~\ref{conj:MainConj} holds.
\end{restatable}

\noindent The proof is based on a compression of the \HHL model that is achieved by a sequence of intermediate fillings that we call mixed nonattacking fillings; see Section~\ref{section:MNF} for more information. Then, we define a family of operators, denoted $T_{i}{j}$, that act on pairs of consecutive boxes in the diagram of a composition $\mu$, and we prove a series of technical results that allow us to track the change in $q,t$-weight after applying one of these operators; for more details, see Section~\ref{section:Operator}. We use these operators to pass from the \HHL model to the restricted nonattacking model via mixed nonattacking fillings; see Section~\ref{section:Compressing}. These operators are similar to the ones defined in~\cite{CHMMW22}.

While working on this paper, Yu Jin and Lin sketched a proof of a similar result in their paper~\cite{JL26}. The primary differences between our results and theirs are the set of fillings and the proof method. The set of restricted nonattacking fillings that we describe in this paper is more restrictive, resulting in our model having fewer terms for general compositions. For the partition case, we get the same number of terms and the same weight. As for the proof methods, we do not use \emph{superization}, which is a method by which the set of allowable entries is extended, and the $q,t$-weight may be negative depending on some details that we do not go into. Due to the similarity between our results and theirs, we believe it is possible to prove Conjecture~\ref{conj:MainConj} using their techniques.

This paper is organized as follows. In Section~\ref{sect:Background}, we review the nonattacking fillings model and the multiline queue model. In Section~\ref{section:NM}, we introduce mixed nonattacking fillings and restricted nonattacking fillings. Section~\ref{section:Operator} introduces the operators $\T{i}{j}$ that act on mixed nonattacking fillings, and we study their effect on the major index and coinversions statistics. Section~\ref{section:Compressing} contains the proof of the compression formula for the nontrivial subset of compositions mentioned above. Finally, in Section~\ref{sect:Conclusion}, we explain why it is challenging to extend this method to more than two columns, and we share the cases for which Conjecture~\ref{conj:MainConj} has been verified using Sage. 

\section{Background}\label{sect:Background}

In this section, we introduce the two models studied in this paper. We begin by reviewing the nonattacking fillings model 
introduced by Haglund, Haiman, and Loehr in 2008. See~\cite{HHL08} for more details and proofs. Then, we move on to the multiline queue model that was introduced by Corteel, Mandelshtam, and Williams in 2018. See~\cite{CMW22} for more details and proofs. 

\subsection{Nonattacking Fillings} \label{subsect:NFM}

A \emph{(weak) composition} is a sequence $\mu = (\mu_1, \ldots, \mu_n)$ of nonnegative integers. The $\mu_i$'s are referred to as the parts of $\mu$, and we define the length of $\mu$, denoted $\ell(\mu)$, to be the number of parts. If the parts form a weakly decreasing sequence, then the composition is called a \emph{partition}, which we denote by $\lambda = (\lambda_1 , \ldots, \lambda_n)$ where $\lambda_1 \ge \ldots \lambda_n \ge 0$. 

Given a composition $\mu$, we define the \emph{diagram of $\mu$}, $\dg'(\mu)$, to be the set of coordinate pairs $(i,j)$, with $1 \le i \le \ell(\mu)$ and $1 \le j \le \mu_i$. A diagram is represented by $\mu_i$ boxes in column $i$. The row numbers increase from bottom to top, and the column numbers increase left to right like Cartesian coordinates. The coordinates $(i,j)$ of a box indicate the column and row of a box. The prime on $\dg$ is there to differentiate it from the convention where the parts of $\mu$ form rows instead of columns.

The \emph{augmented diagram of $\mu$}, $\wdg(\mu)$, is obtained from the diagram $\dg'(\mu)$ by adding the boxes $(i,0)$ for $1 \le i \le \ell(\mu)$. This augmented row is called the \emph{basement}. A pair of boxes is said \emph{to attack each other} if they are in the same row, or they are in consecutive rows where the box in the lower row is to the right of the box in the upper row. If $u$ is a box in the diagram of $\mu$, then we denote the box directly beneath $u$ by $d(u)$. Note that $d(u)$ is in the augmented diagram. In Figure~\ref{fig:ExampleOfDiagrams}, we show the diagram and augmented diagram of $\mu = (0,2,2,2,0,4,2,4)$.

\begin{figure}[h!]
    \centering
    \begin{tikzpicture}[scale = 0.85]
        \begin{scope}
        \draw [thick] (0,1) -- (9,1);
        \draw [thick] (1,2) -- (4,2);
        \draw [thick] (5,2) -- (8,2);
        \draw [thick] (1,3) -- (4,3);
        \draw [thick] (5,3) -- (8,3);
        \draw [thick] (5,4) -- (6,4);
        \draw [thick] (7,4) -- (8,4);
        \draw [thick] (5,5) -- (6,5);
        \draw [thick] (7,5) -- (8,5);

        \draw [thick] (1,1) -- (1,3);
        \draw [thick] (2,1) -- (2,3);
        \draw [thick] (3,1) -- (3,3);
        \draw [thick] (4,1) -- (4,3);
        \draw [thick] (5,1) -- (5,5);
        \draw [thick] (6,1) -- (6,5);
        \draw [thick] (7,1) -- (7,5);
        \draw [thick] (8,1) -- (8,5);
        \end{scope}
        \begin{scope}[xshift = 11cm]

        \draw [thick] (0,0) -- (8,0);
        
        \draw [thick] (-0.5,1) -- (8.5,1);
        \draw [thick] (1,2) -- (4,2);
        \draw [thick] (5,2) -- (8,2);
        \draw [thick] (1,3) -- (4,3);
        \draw [thick] (5,3) -- (8,3);
        \draw [thick] (5,4) -- (6,4);
        \draw [thick] (7,4) -- (8,4);
        \draw [thick] (5,5) -- (6,5);
        \draw [thick] (7,5) -- (8,5);

        \draw [thick] (0,0) -- (0,1);
        \draw [thick] (1,0) -- (1,3);
        \draw [thick] (2,0) -- (2,3);
        \draw [thick] (3,0) -- (3,3);
        \draw [thick] (4,0) -- (4,3);
        \draw [thick] (5,0) -- (5,5);
        \draw [thick] (6,0) -- (6,5);
        \draw [thick] (7,0) -- (7,5);
        \draw [thick] (8,0) -- (8,5);

        \node at (5.5,2.5) {$\mathbf{u}$};
        \node at (5.5,1.5) {$\mathbf{d(u)}$};

        \fill[red!20] (1,2) rectangle (4,3);
        \draw[thick] (1,2) grid (4,3);
        \fill[red!20] (6,2) rectangle (8,3);
        \draw[thick] (6,2) grid (8,3);
        \fill[red!20] (6,1) rectangle (8,2);
        \draw[thick] (6,1) grid (8,2);
        
        \end{scope}
    \end{tikzpicture}
    \caption{The diagram $\dg'(\mu)$ (left) and the augmented diagram $\wdg(\mu)$ (right) of $\mu = (0,2,2,2,0,4,2,4)$. The red boxes attack 
    $u$ and $d(u)$ is the box below $u$.}
    \label{fig:ExampleOfDiagrams}
\end{figure}
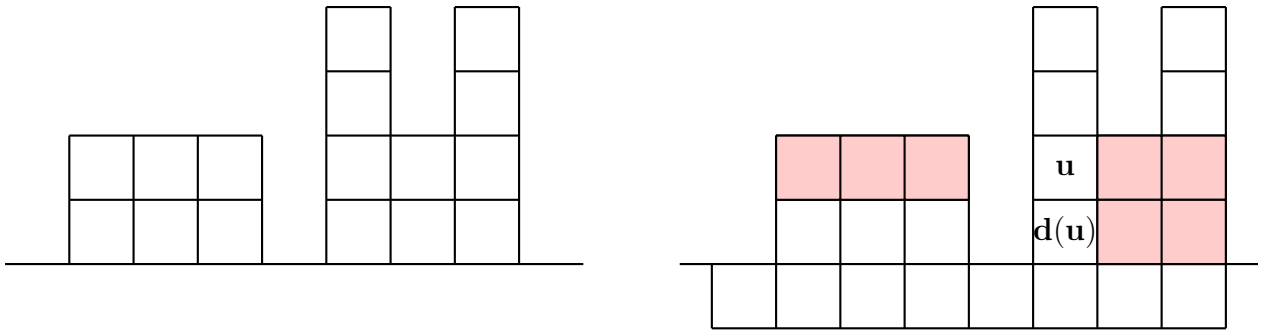

Given a box $u$ in the diagram of $\mu$, we define four sets of boxes. The \emph{leg of $u$} is the set of boxes directly above $u$, and we write $\leg(\mu)$ to denote this set. The \emph{left arm of $u$} is the set of boxes in row $j$ to the left of $u$ whose column is as tall as the column containing $\mu$ or shorter, denoted by $\armL(u)$. The \emph{right arm of $u$} is the set of boxes in the row below $u$ that are to the right of $u$ and whose columns are strictly shorter than the column containing $u$, and we write $\armR(u)$ for this set. Finally, the \emph{arm of $u$}, denoted $\arm(u)$, is the union of the left arm of $u$ and the right arm of $u$. We define $l(u):= |\leg(u)|$ and $a(u) := |\arm(u)|$. In Figure~\ref{fig:ExOfTheFourSets}, we show the various sets for the box $u = (6,2)$ in the augmented diagram of $\mu = (0,2,2,2,0,4,2,4)$.

\begin{figure}[h!]
    \centering
    \begin{tikzpicture}[scale=0.75]
        \draw [thick] (0,0) -- (8,0);
        
        \draw [thick] (-1,1) -- (9,1);
        \draw [thick] (1,2) -- (4,2);
        \draw [thick] (5,2) -- (8,2);
        \draw [thick] (1,3) -- (4,3);
        \draw [thick] (5,3) -- (8,3);
        \draw [thick] (5,4) -- (6,4);
        \draw [thick] (7,4) -- (8,4);
        \draw [thick] (5,5) -- (6,5);
        \draw [thick] (7,5) -- (8,5);

        \draw [thick] (0,0) -- (0,1);
        \draw [thick] (1,0) -- (1,3);
        \draw [thick] (2,0) -- (2,3);
        \draw [thick] (3,0) -- (3,3);
        \draw [thick] (4,0) -- (4,3);
        \draw [thick] (5,0) -- (5,5);
        \draw [thick] (6,0) -- (6,5);
        \draw [thick] (7,0) -- (7,5);
        \draw [thick] (8,0) -- (8,5);

        \node at (5.5,2.5) {$\mathbf{u}$};
        \fill[blue!20] (5,3) rectangle (6,5);
        \draw[thick] (5,3) grid (6,5);
        \fill[red!20] (1,2) rectangle (4,3);
        \draw[thick] (1,2) grid (4,3);
        \fill[green!20] (6,1) rectangle (7,2);
        \draw[thick] (6,1) grid (7,2);
        
    \end{tikzpicture}
    \caption{Depiction of sets $\leg(u)$ in blue, $\armL(u)$ in red, and $\armR(u)$ in green for $u = (6,2)$.}
    \label{fig:ExOfTheFourSets}
\end{figure}
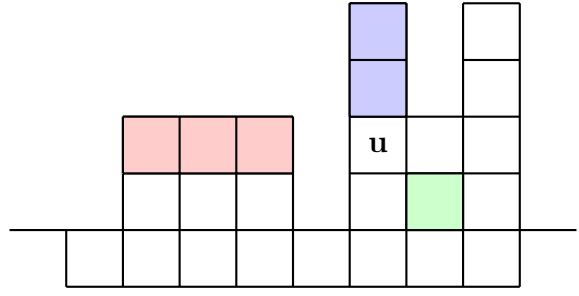

Let $[n] := \{1, 2, \ldots, n\}$. Given a composition $\mu$, we define a \emph{filling of the diagram of $\mu$} as a function $\sigma: \dg'(\mu) \to [n]$, and we extend a filling to the augmented diagram by defining a function $\ws: \wdg(\mu) \to [n]$, so that $\ws(u) = \sigma(u)$ for $u \in \dg'(\mu)$ and $\ws(i,0) = i$ for all $1 \le i \le \ell(\mu)$. In both cases, we refer to $\sigma$ and $\ws$ as fillings, and the different notation indicates whether the filling is of $\dg'(\mu)$ or $\wdg(\mu)$. We say a filling $\ws$ is a \emph{nonattacking filling} if $\ws(u) \neq \ws(v)$ whenever $u$ and $v$ attack each other. We say that $\sigma$ is a \emph{nonattacking filling} if its extension to the augmented diagram is a nonattacking filling. The set of all nonattacking fillings is denoted $\NAF(\mu)$. In Figure~\ref{fig:ExOfFillings} we show an example of a filling that is attacking (left) and a filling that is nonattacking (right). The filling on the left is attacking because of two situations: (1) there are two 5s in row 1, and (2) the 1 in row 2 is to the left of the 1 in row 1. 

\begin{figure}[h!]
    \centering
    \begin{tikzpicture}[scale=0.75]
    \begin{scope}
        \node[] at (4,-1) {Attacking};
        \draw [thick] (0,0) -- (8,0);
        
        \draw [thick] (0,1) -- (8,1);
        \draw [thick] (1,2) -- (4,2);
        \draw [thick] (5,2) -- (8,2);
        \draw [thick] (1,3) -- (4,3);
        \draw [thick] (5,3) -- (8,3);
        \draw [thick] (5,4) -- (6,4);
        \draw [thick] (7,4) -- (8,4);
        \draw [thick] (5,5) -- (6,5);
        \draw [thick] (7,5) -- (8,5);

        \draw [thick] (0,0) -- (0,1);
        \draw [thick] (1,0) -- (1,3);
        \draw [thick] (2,0) -- (2,3);
        \draw [thick] (3,0) -- (3,3);
        \draw [thick] (4,0) -- (4,3);
        \draw [thick] (5,0) -- (5,5);
        \draw [thick] (6,0) -- (6,5);
        \draw [thick] (7,0) -- (7,5);
        \draw [thick] (8,0) -- (8,5);

        \node at (0.5,0.5) {$1$};
        \node at (1.5,0.5) {$2$};
        \node at (2.5,0.5) {$3$};
        \node at (3.5,0.5) {$4$};
        \node at (4.5,0.5) {$5$};
        \node at (5.5,0.5) {$6$};
        \node at (6.5,0.5) {$7$};
        \node at (7.5,0.5) {$8$};

        \node at (1.5,1.5) {$7$};
        \node at (2.5,1.5) {$3$};
        \node[color=blue] at (3.5,1.5) {$\mathbf{1}$};
        \node[color=red] at (5.5,1.5) {$\mathbf{5}$};
        \node at (6.5,1.5) {$6$};
        \node[color=red] at (7.5,1.5) {$\mathbf{5}$};

        \node[color=blue] at (1.5,2.5) {$\mathbf{1}$};
        \node at (2.5,2.5) {$3$};
        \node at (3.5,2.5) {$2$};
        \node at (5.5,2.5) {$4$};
        \node at (6.5,2.5) {$6$};
        \node at (7.5,2.5) {$8$};

        \node at (5.5,3.5) {$1$};
        \node at (7.5,3.5) {$8$};

        \node at (5.5,4.5) {$2$};
        \node at (7.5,4.5) {$3$};
        
        \end{scope}
        \begin{scope}[xshift = 10cm]
        \node[] at (4,-1) {nonattacking};
        
        \draw [thick] (0,0) -- (8,0);
        
        \draw [thick] (0,1) -- (8,1);
        \draw [thick] (1,2) -- (4,2);
        \draw [thick] (5,2) -- (8,2);
        \draw [thick] (1,3) -- (4,3);
        \draw [thick] (5,3) -- (8,3);
        \draw [thick] (5,4) -- (6,4);
        \draw [thick] (7,4) -- (8,4);
        \draw [thick] (5,5) -- (6,5);
        \draw [thick] (7,5) -- (8,5);

        \draw [thick] (0,0) -- (0,1);
        \draw [thick] (1,0) -- (1,3);
        \draw [thick] (2,0) -- (2,3);
        \draw [thick] (3,0) -- (3,3);
        \draw [thick] (4,0) -- (4,3);
        \draw [thick] (5,0) -- (5,5);
        \draw [thick] (6,0) -- (6,5);
        \draw [thick] (7,0) -- (7,5);
        \draw [thick] (8,0) -- (8,5);

        \node at (0.5,0.5) {$1$};
        \node at (1.5,0.5) {$2$};
        \node at (2.5,0.5) {$3$};
        \node at (3.5,0.5) {$4$};
        \node at (4.5,0.5) {$5$};
        \node at (5.5,0.5) {$6$};
        \node at (6.5,0.5) {$7$};
        \node at (7.5,0.5) {$8$};

        \node at (1.5,1.5) {$2$};
        \node at (2.5,1.5) {$3$};
        \node at (3.5,1.5) {$1$};
        \node at (5.5,1.5) {$5$};
        \node at (6.5,1.5) {$6$};
        \node at (7.5,1.5) {$7$};

        \node at (1.5,2.5) {$4$};
        \node at (2.5,2.5) {$3$};
        \node at (3.5,2.5) {$2$};
        \node at (5.5,2.5) {$1$};
        \node at (6.5,2.5) {$6$};
        \node at (7.5,2.5) {$8$};

        \node at (5.5,3.5) {$2$};
        \node at (7.5,3.5) {$8$};

        \node at (5.5,4.5) {$2$};
        \node at (7.5,4.5) {$3$};
        
        \end{scope}
    \end{tikzpicture}
    \caption{The filling on the left is not a nonattacking filling (because of the boxes highlighted in blue and red), while the filling on the right is nonattacking.}
    \label{fig:ExOfFillings}
\end{figure}
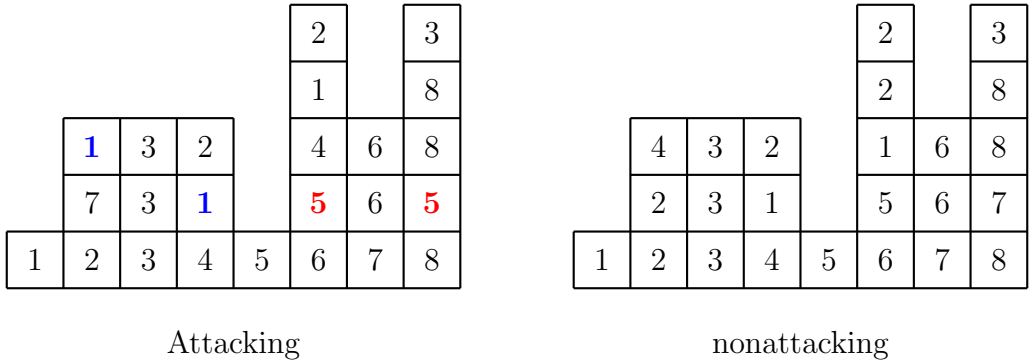

Given a nonattacking filling $\ws$ of $\wdg(\mu)$, we define several important statistics. A box $u \in \dg'(\mu)$ is called a \emph{descent} if $\ws(u) > \ws(d(u))$, and we write $\Des(\ws)$ for the set of all descents. The \emph{major index} of $\ws$ is then defined to be
$$\maj(\ws):= \sum_{\ws \in \Des(\ws)} (l(u) + 1).$$ 
A \emph{type I triple} is a set of three boxes $\{v,u,d(u)\}$ such that $v$ and $d(u)$ are in the same row, the left box is $d(u)$, and the column containing $v$ is strictly shorter than the column containing $u$ and $d(u)$. This is depicted by the three boxes on the left in Figure~\ref{fig:DifferentTriples}. A \emph{type II triple} is a set of three boxes $\{v,u,d(u)\}$, so that $v$ and $u$ are in the same row, the left box is $v$, and the column containing $v$ is either the same height as the column containing $u$ or shorter. This is shown by the three boxes on the right in Figure~\ref{fig:DifferentTriples}. A type I (resp. type II) triple is called a type I (resp. type II)  \emph{coinversion triple} if $\ws(u) < \ws(v) < \ws(d(u))$, $\ws(d(u)) < \ws(v) < \ws(u)$, or $\ws(u) < \ws(d(u)) < \ws(v)$. We define  $\coinv(\ws)$ to be the number of type I and type II coinversion triples of $\ws$. In Figure~\ref{fig:DifferentTriples}, we depict Type I and Type II coinversion triples, where the direction of the arrow indicates the direction in which the entries are increasing.
\begin{figure}[h!]
    \centering
    \begin{tikzpicture}[scale=0.75]
        \begin{scope}
            \node at (1.5,-0.5) {Type I};
        
            \draw[thick] (0,0) grid (1,2);
            \draw[thick] (2,0) grid (3,1);

            \draw[thick, ->] (1.25,1) arc (180:-90:0.25cm);

            \node[] at (0.5,0.5) {$d(u)$};
            \node[] at (0.5,1.5) {$u$};
            \node[] at (2.5,0.5) {$v$};
        \end{scope}
        \begin{scope}[xshift = 6cm]
            \node at (1.5,-0.5) {Type II};
            
            \draw[thick] (0,1) grid (1,2);
            \draw[thick] (2,0) grid (3,2);

            \draw[thick, <-] (1.5,1.25) arc (90:-180:0.25cm);

            \node[] at (2.5,0.5) {$d(u)$};
            \node[] at (2.5,1.5) {$u$};
            \node[] at (0.5,1.5) {$v$};
        \end{scope}
    \end{tikzpicture}
    \caption{Representation of the type I (left) and type II (right) coinversion triples}
    \label{fig:DifferentTriples}
\end{figure}
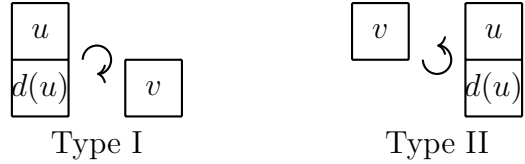

As an example, consider the nonattacking filling in Figure~\ref{fig:ExOfFillings}. The set of descents is $\{(2,2),(4,2),(8,2),(6,4)\}$; hence, $\maj(\ws) = 6$. The set of type I coinversion triple is $\{(6,2),(5,3),(5,2)\}$ while a some of the type II coinversion triples are $$\{(2,1),(4,1),(4,0)\}, \{(2,2),(4,2),(4,1)\}, \{(2,2),(6,2),(6,1)\},  \text{and} \{(7,2),(8,2),(8,1)\}.$$ There are 13 coinversion triples in total. 

We define the $q,t$-weight of a nonattacking filling using these statistics as follows:
\begin{equation*}
    \wtqt(\ws) = q^{\maj(\ws)}t^{\coinv(\ws)}
    \prod_{\substack{u \in \dg'(\mu) \\ \ws(u) \neq \ws(d(u))}} 
    \frac{1-t}{1-q^{l(u)+1}t^{a(u)+1}}.
\end{equation*}
Additionally, we associate each filling $\sigma$ with a monomial $x^\sigma = x_1^{k_1} \ldots x_n^{k_n}$, where $k_i$ is the number of boxes $u \in dg'(\mu)$ such that $\sigma(u) = i$. For example, the nonattacking filling in Figure~\ref{fig:ExOfFillings} has $q,t$-weight:
$$q^{6}t^{13} \frac{1-t}{1-q^2t^4} \frac{1-t}{1-q^4t^5} \frac{1-t}{1-q^4t^6} \frac{1-t}{1-qt} \frac{1-t}{1-qt^3}\frac{1-t}{1-q^3t^5} \frac{1-t}{1-q^3t^6} \frac{1-t}{1-q^2t^2} \frac{1-t}{1-qt^2}$$
and monomial 
$$x^{\sigma} = x_1^2x_2^4x_3^3x_4x_5x_6^2x_7x_8^2.$$

Haglund, Haiman, and Loehr proved that the set of nonattacking fillings of a composition $\mu$, together with their $q,t$-weights and $x^\sigma$ monomials, is a combinatorial model for nonsymmetric Macdonald polynomials. 
\begin{theorem}\cite[Theorem 3.5.1]{HHL08}
    For a composition $\mu$,
    \begin{equation*}
        E_\mu = \sum_{\substack{\sigma:\mu \to [n] \\
        \text{nonattacking}}} \wtqt(\ws) x^\sigma .
    \end{equation*}
\end{theorem}

\subsection{Multiline Queue Model} \label{subsect: MQM}

Let $L$ and $n$ be positive integers. An $L \times n$ \emph{ball system} $B$ is an array where each position is either empty or occupied by a ball, the number of balls weakly decreases from bottom to top, and there is at least one ball in row $L$. Thus, $L$ is the number of rows and $n$ is the number of columns of the array. The rows are labeled 1 through $L$ from bottom to top. See Figure~\ref{fig:BallSystem} for an example of a $4 \times 7$ ball system with 12 balls. 

\begin{figure}[h!]
    \centering
    \begin{tikzpicture}[scale=0.75]
        
        \draw[] (0.5,0.35) circle (0.3cm);
        \draw[] (1.5,0.35) circle (0.3cm);
        \draw[fill=black] (2.5,0.35) circle (0.05cm);
        \draw[] (3.5,0.35) circle (0.3cm);
        \draw[fill=black] (4.5,0.35) circle (0.05cm);
        \draw[fill=black] (5.5,0.35) circle (0.05cm);
        \draw[] (6.5,0.35) circle (0.3cm);

        \draw[fill=black] (.5,1.5) circle (0.05cm);
        \draw[] (1.5,1.5) circle (0.3cm);
        \draw[fill=black] (2.5,1.5) circle (0.05cm);
        \draw[fill=black] (3.5,1.5) circle (0.05cm);
        \draw[] (4.5,1.5) circle (0.3cm);
        \draw[] (5.5,1.5) circle (0.3cm);
        \draw[fill=black] (6.5,1.5) circle (0.05cm);

        \draw[fill=black] (.5,2.65) circle (0.05cm);
        \draw[fill=black] (1.5,2.65) circle (0.05cm);
        \draw[fill=black] (2.5,2.65) circle (0.05cm);
        \draw[fill=black] (3.5,2.65) circle (0.05cm);
        \draw[] (4.5,2.65) circle (0.3cm);
        \draw[] (5.5,2.65) circle (0.3cm);
        \draw[] (6.5,2.65) circle (0.3cm);

        \draw[] (0.5,3.8) circle (0.3cm);
        \draw[fill=black] (1.5,3.8) circle (0.05cm);
        \draw[fill=black] (2.5,3.8) circle (0.05cm);
        \draw[] (3.5,3.8) circle (0.3cm);
        \draw[fill=black] (4.5,3.8) circle (0.05cm);
        \draw[fill=black] (5.5,3.8) circle (0.05cm);
        \draw[fill=black] (6.5,3.8) circle (0.05cm);

        \node[] at (-1.5,0.35) {Row $1$};
        \node[] at (-1.5,1.5) {Row $2$};
        \node[] at (-1.5,2.65) {Row $3$};
        \node[] at (-1.5,3.8) {Row $4$};
    \end{tikzpicture}
    \caption{A $4 \times 7$ ball system with 4 balls in row 1, 3 balls 
    in row 2, 3 balls in row 3, and 2 balls in row 4.}
    \label{fig:BallSystem}
\end{figure}
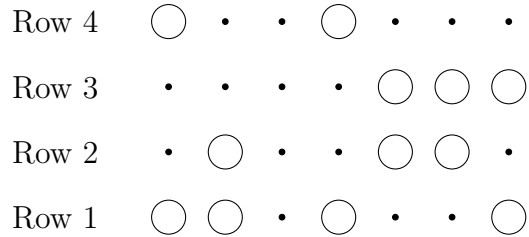

A \emph{multiline queue} of a ball system $B$ is a diagram
where each ball is matched with a ball in the row below and
assigned a label in $[L]$ according to rules that we describe
below. We visualize a pairing in the ball diagram via the
shortest possible strand that either travels straight down or
left to right, wrapping around if necessary.

We define an algorithm that takes in a ball system $B$ and
outputs a multiline queue $Q$. Suppose $B$ is an $L \times n$
ball system. First, each ball in row $L$ is paired with a ball
in row $L-1$ in the following manner: if a ball in row $L$
has a ball directly below it, then we must pair those balls; such a
pairing is called a \emph{trivial pair}; otherwise, a ball in
row $L$ is paired with a ball in row $L-1$ by 
the shortest possible strand that travels left to right,
wrapping around if necessary. Each ball paired at this step
is labeled by an $L$. Then, we proceed to repeat the previous
step, maintaining the same requirements for trivial pairs,
for the balls labeled by an $L$ in row $L-1$, and
labeling the balls paired at this step with an $L$. We
repeat this for each ball labeled by an $L$, decreasing by
one row at each step, ending once each ball in row $L$ is
connected to a ball in row 1. Next, we repeat the same 
procedure that we used with the balls in row $L$, for each
unpaired ball in row $L-1$, but instead we label each ball
paired at this step with an $L-1$. We continue in this way
for each set of unpaired balls in rows $L, L-1, \ldots, 2$.
Finally, any remaining unpaired ball in row $1$ is labeled by a 1. This completes the algorithm. The result is a
series of strands paired with the last ball in row 1.

\begin{remark}
    A ball $b$ in a pairing step pairs with a ball $b'$ directly below $b$ if and only if $b'$ is free. 
\end{remark}

The labels of the balls in row 1 form a composition $\mu
= (\mu_1, \ldots, \mu_n)$, where $\mu_i$ is either the label
of the ball in row 1 and column $i$ or $\mu_i = 0$ if there
is no such ball. Then, $\mu$ is the \emph{type of a multiline
queue} and the multiline queue is said to be of \emph{type
$\mu$}. The set of all multiline queues of type $\mu$ is 
denoted $\MLQ(\mu)$. In Figure~\ref{fig:MLQExample}, we show
an example of a multiline queue with type $\mu = 
(2,4,0,1,0,0,4)$. There are two trivial pairs in the example:
a pair of balls labeled 4 in column 2, and a pair of balls
labeled 4 in column 6.

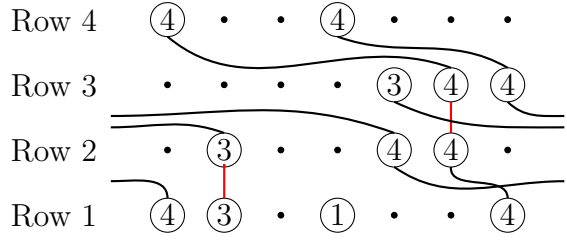
\begin{figure}[h!]
    \centering
    \begin{tikzpicture}[scale=0.75]
        
        \draw[] (0.5,0.35) circle (0.3cm);
        \draw[] (1.5,0.35) circle (0.3cm);
        \draw[fill=black] (2.5,0.35) circle (0.05cm);
        \draw[] (3.5,0.35) circle (0.3cm);
        \draw[fill=black] (4.5,0.35) circle (0.05cm);
        \draw[fill=black] (5.5,0.35) circle (0.05cm);
        \draw[] (6.5,0.35) circle (0.3cm);

        \draw[fill=black] (.5,1.5) circle (0.05cm);
        \draw[] (1.5,1.5) circle (0.3cm);
        \draw[fill=black] (2.5,1.5) circle (0.05cm);
        \draw[fill=black] (3.5,1.5) circle (0.05cm);
        \draw[] (4.5,1.5) circle (0.3cm);
        \draw[] (5.5,1.5) circle (0.3cm);
        \draw[fill=black] (6.5,1.5) circle (0.05cm);

        \draw[fill=black] (.5,2.65) circle (0.05cm);
        \draw[fill=black] (1.5,2.65) circle (0.05cm);
        \draw[fill=black] (2.5,2.65) circle (0.05cm);
        \draw[fill=black] (3.5,2.65) circle (0.05cm);
        \draw[] (4.5,2.65) circle (0.3cm);
        \draw[] (5.5,2.65) circle (0.3cm);
        \draw[] (6.5,2.65) circle (0.3cm);

        \draw[] (0.5,3.8) circle (0.3cm);
        \draw[fill=black] (1.5,3.8) circle (0.05cm);
        \draw[fill=black] (2.5,3.8) circle (0.05cm);
        \draw[] (3.5,3.8) circle (0.3cm);
        \draw[fill=black] (4.5,3.8) circle (0.05cm);
        \draw[fill=black] (5.5,3.8) circle (0.05cm);
        \draw[fill=black] (6.5,3.8) circle (0.05cm);

        \node[] at (-1.5,0.35) {Row $1$};
        \node[] at (-1.5,1.5) {Row $2$};
        \node[] at (-1.5,2.65) {Row $3$};
        \node[] at (-1.5,3.8) {Row $4$};

        \draw[thick, -] (0.5, 3.5) to[out=-40,in=160] (5.5,2.95);
        \draw[thick, -] (3.5, 3.5) to[out=-20,in=150] (6.5,2.95);
        \node[] at (0.5,3.8) {$4$};
        \node[] at (3.5,3.8) {$4$};
        \node[] at (5.5,2.65) {$4$};
        \node[] at (6.5,2.65) {$4$};
        \node[] at(4.5,2.65) {$3$};

        \draw[color=red, thick, -] (5.5, 2.35) to[out=-90,in=90] (5.5,1.8);
        \draw[thick, -] (6.5, 2.35) to[out=-45,in=180] (7.5, 2.1);
        \draw[thick,-] (-0.5, 2.1) to[out=0, in=160] (4.5,1.8);
        \draw[thick, -] (4.5, 2.35) to[out=-25,in=180] (7.5,1.9);
        \draw[thick,-] (-0.5, 1.9) to[out=0, in=160] (1.5,1.8);
        \node[] at (1.5,1.5) {$3$};
        \node[] at (4.5,1.5) {$4$};
        \node[] at (5.5,1.5) {$4$};

        \draw[red, thick, -] (1.5,1.2) to[out=--90,in=90] (1.5,0.65);
        \draw[thick, -] (5.5,1.2) to[out=-95,in=90] (6.5,0.65);
        \draw[thick, -] (4.5,1.2) to[out=-35, in=180] (7.5,0.95);
        \draw[thick, -] (-0.5,0.95) to[out=0, in=90] (0.5,0.65);
        \node[] at (0.5,0.35) {$4$};
        \node[] at (1.5,0.35) {$3$};
        \node[] at (3.5,0.35) {$1$};
        \node[] at (6.5,0.35) {$4$};
    \end{tikzpicture}
    \caption{A multiline queue of type $(4,3,0,1,0,0,4)$, with the two trivial pairs depicted with red strands.}
    \label{fig:MLQExample}
\end{figure}

Given a multiline queue $Q$ of type $\mu$, the $q,t$-weight of $Q$ is defined using the following process. Consider the nontrivial pairings between rows $r$ and $r-1$, for $1 \le r \le L$. We read the balls in decreasing order of their labels and right to left within a label. We imagine placing strands one at a time following the reading order. An unpaired ball is considered \emph{free}. Trivially paired balls are not considered free. If ball $b$ in column $c$ is paired to ball $b'$ in column $c'$, any free ball that the strand passes over is said to be \emph{skipped}. Let $i$ be the label of the pair $p$ of balls $b$ and $b'$, then 
\[   
\wtqt(p) := 
     \begin{cases}
       \dfrac{(1-t)t^{\#\text{skipped}}}{1-q^{i-
       r+1}t^{\#\text{free}}} \cdot q^{i-
       r+1} & \quad \text{ if } c' < c,\\[0.4cm]
       \dfrac{(1-t)t^{\#\text{skipped}}}{1-q^{i
       -r+1}t^{\#\text{free}}} & \quad \text{ if } c' > c. \\
     \end{cases}
\]

\begin{remark}
    Trivial pairs do not contribute to the $q,t$-weight.
\end{remark}

The $q,t$-weight for $Q$ is the following product over
all pairs $p$ in the multiline queue:
\[
\wtqt(Q) := \prod_p \wtqt(p).
\]
Then the \emph{weight} of a multiline queue $Q$ is 
\[
\wt(Q) = \wtqt(Q)\wtx(Q),
\]
where $\wtx(Q)=x_1^{m_1} \ldots x_r^{m_r}$ and $m_i$ is the
number of balls in column $i$. For example, the weight of the
multiline queue in Figure~\ref{fig:MLQExample} is
$$\wt(Q) = \frac{t(1-t)}{1-qt^2}\frac{1-t}{1-qt}
\frac{q^2(1-t)}{1-q^2t^2}
\frac{1-t}{1-q^3t^3}\frac{q(1-t)}
{1-qt^2}x_1^2x_2^2x_4^2x_5x_6x_7^2.$$

Corteel, Mandelshtam, and Williams proved that multiline
queues are, in fact, a combinatorial model for nonsymmetric
Macdonald polynomials for partitions, $\lambda$. 
\begin{lemma}\cite[Lemma 1.26]{CMW22}
    For any partition $\lambda$,
    $$E_\lambda = \sum_{Q \in \MLQ(\lambda)} \wt(Q).$$
\end{lemma}

\section{A new combinatorial model} \label{section:NM}

In Section \ref{section:MNF}, we define mixed nonattacking fillings and mixed $q,t$-weight. Then, in Section \ref{section:RNF}, we define restricted nonattacking fillings and their $q,t$-weight. As mentioned in the introduction, we conjecture that the restricted nonattacking fillings form a combinatorial model for nonsymmetric Macdonald polynomials for all compositions. In Section \ref{section:CFNF}, we outline the proof of this conjecture for a large family of compositions, and we explain the reasoning behind the compression. 

Before we begin, we define what it means for a box in $\dg'(\mu)$ to be a \emph{trivial pair}, the following definition mirrors the definition of a trivial pair in the \MLQ model. 
\begin{definition}
    Suppose $\ws$ is a nonattacking filling of $\wdg(\mu)$. A box $u \in \dg'(\mu)$ is a \emph{trivial pair} of $\ws$ if $u$ and $d(u)$ have the same entry; that is, $\ws(u) = \ws(d(u))$.
\end{definition}
\noindent For example, the trivial pairs of the nonattacking filling boxes in Figure~\ref{fig:ExOfFillings} are boxes $(2,1)$, $(3,1)$, $(3,2)$, $(7,2)$, $(8,3)$, and $(6,4)$. 

\subsection{Mixed nonattacking Fillings} \label{section:MNF}

Recall that $\wdg(\mu)$ is defined to be a set of coordinate pairs $\{(i,j) \mid 1 \le i \le \ell(\mu), 0 \le j \le \lambda_i\}$. We define $\mathcal{P}(\wdg(\mu))$ to be the power set of $\wdg(\mu)$.   
\begin{definition}
    Let $\mu$ be a composition and $R: \dg'(\mu)\to \mathcal{P}(\wdg(\mu))$ be a function. We say that $R$ is a \emph{restriction function} of $\dg'(\mu)$ if, for all $u \in \dg'(\mu)$, the boxes in $R(u)$ are all in the row below and to the left of $u$; that is, if $u = (i,j)$ and $v = (i',j') \in R(u)$, then $i' < i$ and $j' = j-1$. We say that the set $R(u)$  \emph{restricts} $u$, and it is called a \emph{restriction set} of $u$.
\end{definition}
\begin{remark}
    It is possible for $R(u)$ to be equal to the empty set.
\end{remark}
\begin{definition}
    Given a restriction function $R$ of $\dg'(\mu)$, a nonattacking filling $\ws$ is a \emph{mixed nonattacking filling for $R$} if $\ws(u) \neq \ws(v)$ for all $v \in R(u)$. The set of all mixed fillings is denoted by $\MF(\mu,R(u))$, or simply $\MF(\mu)$ if $R$ is clear from context.
\end{definition}

\begin{remark}
Different restriction functions lead to different subsets of nonattacking fillings. For example, if $R(u) = \emptyset$ for all $u \in \dg'(\mu)$, then $\MF(\mu)$ is the set of all nonattacking fillings. However, if $R(u)$ is non-empty for some $u \in \dg'(\mu)$, then the set of mixed fillings is a strict subset of the set of nonattacking fillings.
\end{remark}

In Figure~\ref{fig:ExMixedNF}, we show an example of a mixed nonattacking filling of $\mu = (0,2,2,2,0,4,2,4)$,  with the restriction function $R(4,2) = \{(3,1), (2,1)\}$, $R(7,2) = \{(6,1), (3,1)\}$, and $R(u) = \emptyset$ for any other $u \in \dg'(\mu)$.
\begin{figure}[h!]
    \centering
    \begin{tikzpicture}[scale=0.75]
    \begin{scope}
        \draw [thick] (0,0) -- (8,0);
        
        \draw [thick] (0,1) -- (8,1);
        \draw [thick] (1,2) -- (4,2);
        \draw [thick] (5,2) -- (8,2);
        \draw [thick] (1,3) -- (4,3);
        \draw [thick] (5,3) -- (8,3);
        \draw [thick] (5,4) -- (6,4);
        \draw [thick] (7,4) -- (8,4);
        \draw [thick] (5,5) -- (6,5);
        \draw [thick] (7,5) -- (8,5);

        \draw [thick] (0,0) -- (0,1);
        \draw [thick] (1,0) -- (1,3);
        \draw [thick] (2,0) -- (2,3);
        \draw [thick] (3,0) -- (3,3);
        \draw [thick] (4,0) -- (4,3);
        \draw [thick] (5,0) -- (5,5);
        \draw [thick] (6,0) -- (6,5);
        \draw [thick] (7,0) -- (7,5);
        \draw [thick] (8,0) -- (8,5);

        \node at (0.5,0.5) {$1$};
        \node at (1.5,0.5) {$2$};
        \node at (2.5,0.5) {$3$};
        \node at (3.5,0.5) {$4$};
        \node at (4.5,0.5) {$5$};
        \node at (5.5,0.5) {$6$};
        \node at (6.5,0.5) {$7$};
        \node at (7.5,0.5) {$8$};

        \node at (1.5,1.5) {$2$};
        \node at (2.5,1.5) {$3$};
        \node at (3.5,1.5) {$4$};
        \node at (5.5,1.5) {$5$};
        \node at (6.5,1.5) {$6$};
        \node at (7.5,1.5) {$8$};

        \node at (1.5,2.5) {$2$};
        \node at (2.5,2.5) {$1$};
        \node at (3.5,2.5) {$7$};
        \node at (5.5,2.5) {$4$};
        \node at (6.5,2.5) {$6$};
        \node at (7.5,2.5) {$8$};

        \node at (5.5,3.5) {$1$};
        \node at (7.5,3.5) {$8$};

        \node at (5.5,4.5) {$2$};
        \node at (7.5,4.5) {$3$};
        
        \end{scope}
    \end{tikzpicture}
    \caption{Example of a mixed nonattacking for $\mu = (0,2,2,2,0,4,2,4)$ with restriction function $R(4,2) = \{(3,1), (2,1)\}$, $R(7,2) = \{(6,1), (3,1)\}$, and 
$R(u) = \emptyset$ for any other $u \in \dg'(\mu)$.}
    \label{fig:ExMixedNF}
\end{figure}

We define additional statistics on mixed nonattacking fillings to emulate the effect of trivial pairs on multiline queues.
\begin{definition} \label{def:TrivialPairOfBox}
Let $\mu$ be a composition, and let $u\in \dg'(\mu)$. Consider a mixed nonattacking filling $\ws$ for a restriction function $R$. We say that a type II coinversion triple $\{v,u,d(u)\}$ is a \emph{trivial coinversion triple} if $d(v) \in R(u)$ and $v$ is a trivial pair. If $u \in \dg'(\mu)$ such that $\ws(u) \neq \ws(d(u))$, we define a \emph{trivial pair of $u$} as a trivial pair $v$ in $\dg'(\mu)$ such that $d(v) \in R(u)$.Then, denote by $t(u)$ the number of trivial pairs of $u$, and by $r(u)$ the number of trivial coinversion triples of the form $\{v,u,d(u)\}$.
\end{definition}

For example, consider the mixed nonattacking filling shown
in Figure~\ref{fig:ExMixedNF}. The box $u = (4,2)$ has
restriction set $\{(3,1),(2,1)\}$; hence, $t(u) = 1$ and
$r(u) = 1$. The box $v = (7,2)$ has restriction set $\{(6,1)
(6,1)\}$; hence, $t(v) = 0$ and $r(v) = 0$. All other boxes
have empty restriction sets, so they do not have any such
statistics to compute. 

\begin{remark}
    In the example shown in Figure~\ref{fig:ExMixedNF}, 
    all boxes in the restriction sets have a box above, but that is not necessary. It is possible that $v=(i',j-1) \in R(u)$ for some $u$ while $(i',j)$ is not in $\dg'(\mu)$.
\end{remark}

To define the $q,t$-weight of a mixed nonattacking filling,
we first define the $q,t$-weight of $u \in \dg'(\mu)$ with
entry $\ws(u)$. 
\begin{definition}
    Let $\ws$ be a mixed nonattacking filling of $\mu$ and
    with restriction set $R(u)$. Fix a box $u \in dg'(\mu)$. 
    The mixed $q,t$-weight of $u$ is
    \begin{equation*}
    \wt_m(u) = 
    \begin{cases}
        q^{l(u)+1}t^{c(u)-r(u)} \dfrac{1-t}{1-q^{l(u)+1}
        t^{a(u)+1-t(u)}} & \qquad \text{ if } \ws(u) > \ws(d(u)), \\
        t^{c(u)-r(u)} \dfrac{1-t}{1-q^{l(u)+1}
        t^{a(u)+1-t(u)}} & \qquad \text{ if } \ws(u) < \ws(d(u)), \\
    \end{cases}
    \end{equation*}
    where  $c(u)$ denotes the number of coinversion triples
    of the form $\{v,u,d(u)\}$ of either type. 
    Then, the $q,t$-weight of the mixed nonattacking 
    filling is 
    \begin{equation*} 
    \wt_m(\ws) = \prod_{\substack{u \in \dg'(\mu) \\ \ws(u) \neq
    \ws(d(u))}} \wt_{m}(u),
    \end{equation*}
    where the product is over all boxes in $\dg'(\mu)$.
\end{definition}

\begin{remark}
    If $R(u) = \emptyset$ for all $u \in \dg'(\mu)$, the mixed $q,t$-weight defined above coincides with the usual $q,t$-weight of a nonattacking filling (see Section~\ref{subsect:NFM}), and so we recover the nonattacking filling model of $\HHL$.
\end{remark}

Continuing with the mixed nonattacking filling shown in Figure~\ref{fig:ExMixedNF}, consider the box $u = (4,2)$. The 
$q,t$-weight of $u$ is 
$$\wt_m(u) = \frac{qt(1-t)}{1-qt^2},$$
and all the other boxes have the same weight as they have in the original nonattacking filling model; thus,
the $q,t$-weight of the mixed nonattacking filling is
$$\wt_m(\ws) = \frac{1-t}{(1-q^4t^5)}\frac{1-t}{1-q^2t^4}\frac{1-t}{1-qt}
\frac{qt(1-t)}{1-qt^2}\frac{1-t}{1-q^3t^5}\frac{1-t}{1-q^2t^2}\frac{q(1-t)}{1-qt}
\frac{1-t}{1-qt^2}.$$ 

\subsection{Restricted nonattacking Fillings} \label{section:RNF}

Let $\mu$ be a composition. We define the following
set of boxes for each $u \in \dg'(\mu)$.

\begin{definition} \label{def:Hit} 
    Given $u = (i,j) \in \dg'(\mu)$ a box $v = (i',j-1) \in dg'(\mu)$
    \emph{hits} $u$ if 
    \begin{enumerate}
        \item $i' < i$,
        \item $\mu_{i'} = \mu_{i}$, and
        \item $\mu_c = \mu_i$ or $\mu_c < j-1$ for all $i' < c < i$.
    \end{enumerate}
\end{definition}
In Figure~\ref{fig:ExHit}, we show the boxes that hit $u = (4,1)$ in blue, $v = (4,2)$ in
pink and blue, $w = (8,4)$ in green, all other $v \in \dg'(\mu)$ are not hit by any boxes.

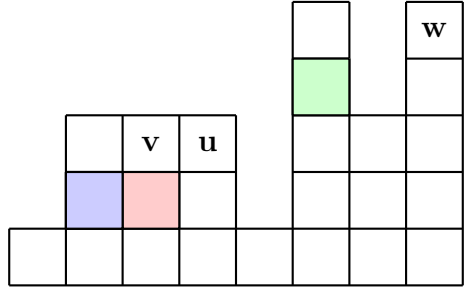
\begin{figure}[H]
    \centering
    \begin{tikzpicture}[scale=0.75]
        \draw [thick] (0,0) -- (8,0);
        
        \draw [thick] (0,1) -- (8,1);
        \draw [thick] (1,2) -- (4,2);
        \draw [thick] (5,2) -- (8,2);
        \draw [thick] (1,3) -- (4,3);
        \draw [thick] (5,3) -- (8,3);
        \draw [thick] (5,4) -- (6,4);
        \draw [thick] (7,4) -- (8,4);
        \draw [thick] (5,5) -- (6,5);
        \draw [thick] (7,5) -- (8,5);

        \draw [thick] (0,0) -- (0,1);
        \draw [thick] (1,0) -- (1,3);
        \draw [thick] (2,0) -- (2,3);
        \draw [thick] (3,0) -- (3,3);
        \draw [thick] (4,0) -- (4,3);
        \draw [thick] (5,0) -- (5,5);
        \draw [thick] (6,0) -- (6,5);
        \draw [thick] (7,0) -- (7,5);
        \draw [thick] (8,0) -- (8,5);

        \node at (3.5,2.5) {$\mathbf{u}$};
        \fill[red!20] (1,1) rectangle (3,2);
        \draw[thick] (1,1) grid (3,2);
        \node at (2.5,2.5) {$\mathbf{v}$};
        \fill[blue!20] (1,1) rectangle (2,2);
        \draw[thick] (1,1) grid (2,2);
        \node at (7.5,4.5) {$\mathbf{w}$};
        \fill[green!20] (5,3) rectangle (6,4);
        \draw[thick] (5,3) grid (6,4);
    \end{tikzpicture}
    \caption{Depicts the boxes that hit $u$, $v$, and $w$.}
    \label{fig:ExHit}
\end{figure}

We define restricted nonattacking fillings as a specific set
of mixed nonattacking fillings.
\begin{definition} 
Define the restriction function 
$R: \dg'(\mu) \to \mathcal{P}(\wdg (\mu))$ via 
$$
R(u) := \{v \in \wdg(\mu)\ \mid\ v \text{ hits } u \}.
$$
We say that a nonattacking filling is a \emph{restricted nonattacking filling} if it is a
mixed nonattacking filling for $R$. The set of all restricted nonattacking fillings of $\mu$ is denoted by $\mathcal{R}(\mu)$.
\end{definition}

In Figure~\ref{fig:ExRNF}, we show an
example of a restricted nonattacking filling for the partition
$\mu = (2,2,2,0,0)$.
\begin{figure}[h!]
    \centering
    \begin{tikzpicture}[scale=0.75]
        \draw [thick] (0,0) -- (6,0);
        
        \draw [thick] (0,1) -- (6,1);
        \draw [thick] (0,2) -- (4,2);
        \draw [thick] (0,3) -- (4,3);

        \draw [thick] (0,0) -- (0,3);
        \draw [thick] (1,0) -- (1,3);
        \draw [thick] (2,0) -- (2,3);
        \draw [thick] (3,0) -- (3,3);
        \draw [thick] (4,0) -- (4,3);
        \draw [thick] (5,0) -- (5,1);
        \draw [thick] (6,0) -- (6,1);

        \node at (0.5,0.5) {$1$};
        \node at (1.5,0.5) {$2$};
        \node at (2.5,0.5) {$3$};
        \node at (3.5,0.5) {$4$};
        \node at (4.5,0.5) {$5$};
        \node at (5.5,0.5) {$6$};
        
        \node at (0.5,1.5) {$1$};
        \node at (1.5,1.5) {$2$};
        \node at (2.5,1.5) {$3$};
        \node at (3.5,1.5) {$4$};

        \node at (0.5,2.5) {$6$};
        \node at (1.5,2.5) {$2$};
        \node at (2.5,2.5) {$3$};
        \node at (3.5,2.5) {$5$};
    \end{tikzpicture}
    \caption{A restricted nonattacking filling of the augmented diagram 
    of $\mu = (2,2,2,0,0)$.}
    \label{fig:ExRNF}
\end{figure}

In fact, if $\mu$ is a partition, then there is a bijection from restricted nonattacking fillings to multiline queues that is not weight preserving. If Conjecture~\ref{conj:MainConj} is true, then there should be a weight-preserving bijection. Unfortunately, we have not been able to find one yet. 
\begin{prop}
    If $\mu$ is a partition, then the number of restricted nonattacking fillings equals the number of multiline queues.
\end{prop}

\begin{proof}
    Let $Q \in \MLQ(\mu)$ such that $Q$ has type $\mu = (\mu_1, \ldots, \mu_k)$. For each ball in row $1$ and column $i$, there is a sequence of paired balls $b_1^i, b_2^i, \ldots b_{\mu_i}^i$ where $b_j^i$ is in row $j$ and columns $c_1^i = i, c_2^i, \ldots, c_{\mu_i}^i$ respectively. The superscript $i$ denotes the sequence of balls that connect to the $i$th ball in row 1 of the multiline queue. 
    
    We define a map $\ws: \wdg(\mu) \to [n]$ via $\ws(i,j) = c_{j}^i$ and $\ws(i,0) = i$, where $1 \le i \le k$ and $1 \le j \le \mu_i$. We claim that $\ws$ is a restricted nonattacking filling. To prove this claim, we have to show that if $u$ and $v$ are attacking or hitting in $\wdg$, then $\ws(u) \neq \ws(v)$. Notice that each column $c_j^i$ appears in at most one sequence of strands; hence, $\ws(u) \neq \ws(v)$ when $u$ and $v$ are in the same row. It remains to show that, if $v$ is in the row below $u$, then $\ws(u) \neq \ws(v)$ unless $v$ is to the left of $u$ and the column containing $v$ is larger than the column containing $u$.
    
    Suppose $u = (i,j)$ and $v = (i',j-1)$. If there is not a ball in row $j-1$ and column $c_{j}^i$, then $\ws(v) \neq c_{j}^i$ for any box in row $j-1$; hence, $\ws(v) \neq \ws(u)$. Otherwise, assume there is a ball $b'$ directly beneath the ball $b$ in row $j$ and column $c_{j}^i$.    
    There are two possibilities when pairing $b$: first, $b'$ is unpaired or, second, $b'$ has already been paired. In the first case, if $b'$ is unpaired, then $b$ and $b'$ form a trivial pair and those balls must pair. As a result, $\ws(i,j) = \ws(i,j-1)$ and it is not possible for a different box in row $j$ of $\wdg(\mu)$ to have the entry $c_{j}^i$. In the second case, the strand that contains $b'$ must have paired before the pairing step for $b$. This is only possible if the label of $b'$ is strictly greater than the label of $b$. Then the strand containing $b'$ would map to a column $'$ to the left of $i$ such that $\mu_{i'} > \mu_i$. By the definition of a restricted nonattacking filling, $\ws(i,j) = \ws(i',j-1)$ is allowed if $i' < i$ and $\mu_{i'} > \mu_i$. There are no other possible pairings; hence, this shows that $\ws$ is a restricted nonattacking filling.

    The inverse map takes a restricted nonattacking filling $\ws$ and constructs a multiline $Q$ by placing a ball in row $j$ and column $\ws(i,j)$ for each box $(i,j) \in \dg'(\mu)$. Then we pair a ball $b$ in row $j$ and column $\ws(i,j)$ to the ball $b'$ in row $j-1$ and column $\ws(i,j-1)$. The resulting multiline queue $Q$ has type $\mu$, and the map is well-defined.  
\end{proof}

\subsection{Compression Formula for nonattacking Fillings} \label{section:CFNF}

In this section, we present our conjecture that the restricted nonattacking filling model together with the mixed $q,t$-weight forms a combinatorial model for the nonsymmetric Macdonald polynomials. We also give a sketch of the proof of Theorem~\ref{thm:MainTheorem}. Additionally, we discuss the compositions for which we have verified the conjecture. Then, we provide some examples of restricted nonattacking fillings and the corresponding compression of nonattacking fillings in Figures~\ref{fig:RNF(0,2,2,2)} and~\ref{fig:CompressionExample}, respectively. We begin by recalling the conjecture.
\MainConjecture*

This conjecture has been verified by computer for all compositions with at most 6 parts and all parts less than or equal to 3, and for compositions with at most 10 parts and largest part being 1. More generally, we have the following result for a larger family of compositions.

\begin{theorem} \label{thm:MainResult}
Suppose $\mu = (\mu_1, \ldots, \mu_n)$ is a composition such that each box $u \in \dg'(\mu)$ hits at most one box. Then, $$E_{\mu} = \sum_{\ws \in \mathcal{R}(\mu)} \wt_{m}(\ws)x^{\sigma}.$$
\end{theorem}

The compositions $\mu$ included in Theorem~\ref{thm:MainResult} are those that, for any group of three equal parts $\mu_i = \mu_j = \mu_k$ ($i < j < k$), there is at least one part $\mu_c \ge \mu_i - 1$ where $i < c < j$ or $j < c < k$. Thus, the group of three parts is separated from each other, guaranteeing that each box hits at most 1 box. For example, the composition $(0,3,0,3,2,2,4,0,2,2,4)$ satisfies the conditions for Theorem~\ref{thm:MainResult}. However, the composition $(3,3,0,3,2,2,4,0,2,2,4)$ does not satisfy the conditions since box $(4,3)$ is hit by boxes $(1,2)$ and $(2,2)$. 

Before discussing the outline of the proof of our main result, we present an example. In Figure~\ref{fig:RNF(0,2,2,2)} we show all the restricted nonattacking fillings for the composition $\mu = (0,2,2,2)$ together with their mixed $q,t$-weight and corresponding monomial. Note that there are 64 nonattacking fillings and only 16 restricted nonattacking fillings; thus, there are far fewer restricted nonattacking fillings than the general nonattacking fillings. We have verified that the corresponding generating function for this example is equal to $E_{(0,2,2,2)}$.

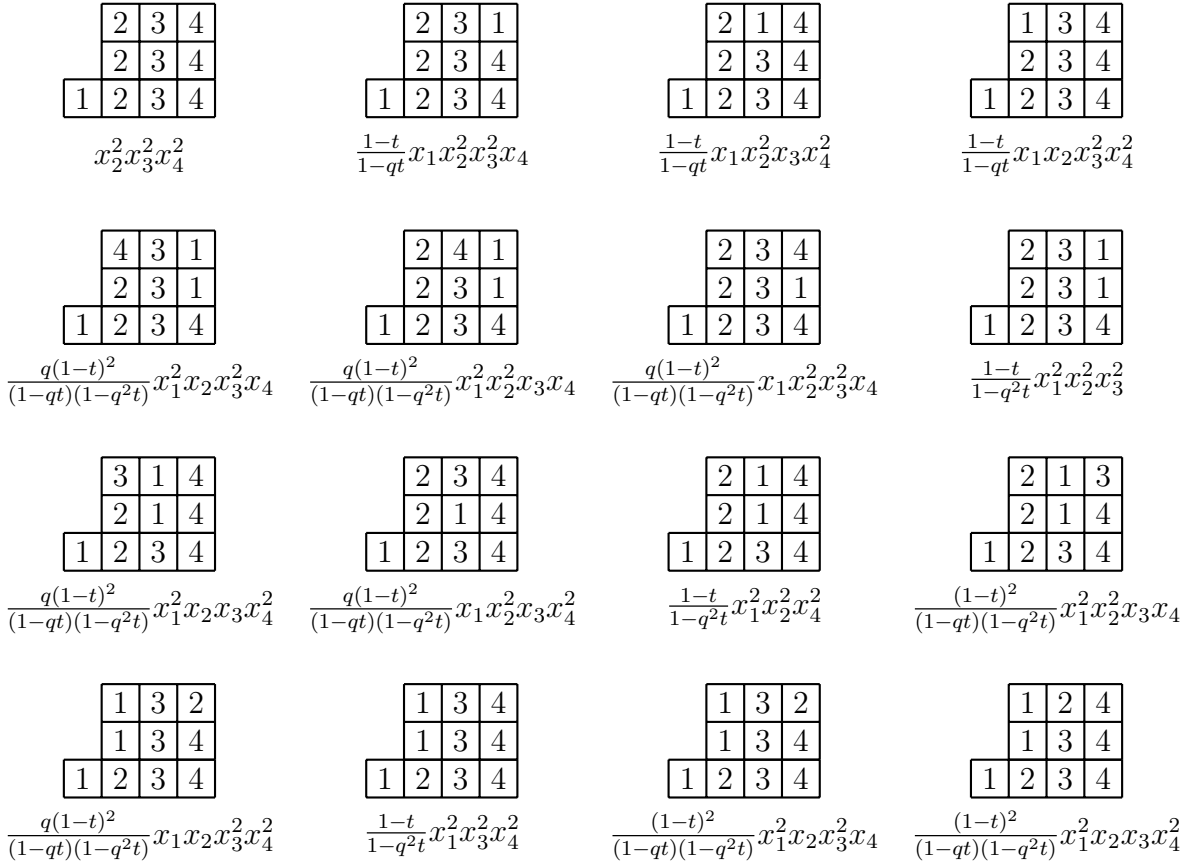
\begin{figure}[ht]
    \centering
    \begin{tikzpicture}[scale=0.5]
        \begin{scope}[xshift=-16cm, yshift=0]
            \draw[thick] (0,0) grid (4,1);
            \draw[thick] (1,1) grid (4,3);

            \node[] at (0.5,0.5) {$1$};
            \node[] at (1.5,0.5) {$2$};
            \node[] at (2.5,0.5) {$3$};
            \node[] at (3.5,0.5) {$4$};

            \node[] at (1.5,1.5) {$2$};
            \node[] at (2.5,1.5) {$3$};
            \node[] at (3.5,1.5) {$4$};

            \node[] at (1.5,2.5) {$2$};
            \node[] at (2.5,2.5) {$3$};
            \node[] at (3.5,2.5) {$4$};

            \node[] at (2,-1) {$x_2^2x_3^2x_4^2$};
        \end{scope}
        \begin{scope}[xshift=-8cm, yshift=0]
            \draw[thick] (0,0) grid (4,1);
            \draw[thick] (1,1) grid (4,3);

            \node[] at (0.5,0.5) {$1$};
            \node[] at (1.5,0.5) {$2$};
            \node[] at (2.5,0.5) {$3$};
            \node[] at (3.5,0.5) {$4$};

            \node[] at (1.5,1.5) {$2$};
            \node[] at (2.5,1.5) {$3$};
            \node[] at (3.5,1.5) {$4$};

            \node[] at (1.5,2.5) {$2$};
            \node[] at (2.5,2.5) {$3$};
            \node[] at (3.5,2.5) {$1$};

            \node[] at (2,-1) {$\frac{1-t}{1-qt}x_1x_2^2x_3^2x_4$};
        \end{scope}
        \begin{scope}[xshift=0cm, yshift=0]
            \draw[thick] (0,0) grid (4,1);
            \draw[thick] (1,1) grid (4,3);

            \node[] at (0.5,0.5) {$1$};
            \node[] at (1.5,0.5) {$2$};
            \node[] at (2.5,0.5) {$3$};
            \node[] at (3.5,0.5) {$4$};

            \node[] at (1.5,1.5) {$2$};
            \node[] at (2.5,1.5) {$3$};
            \node[] at (3.5,1.5) {$4$};

            \node[] at (1.5,2.5) {$2$};
            \node[] at (2.5,2.5) {$1$};
            \node[] at (3.5,2.5) {$4$};

            \node[] at (2,-1) {$\frac{1-t}{1-qt}x_1x_2^2x_3x_4^2$};
        \end{scope}
        \begin{scope}[xshift=8cm, yshift=0]
            \draw[thick] (0,0) grid (4,1);
            \draw[thick] (1,1) grid (4,3);

            \node[] at (0.5,0.5) {$1$};
            \node[] at (1.5,0.5) {$2$};
            \node[] at (2.5,0.5) {$3$};
            \node[] at (3.5,0.5) {$4$};

            \node[] at (1.5,1.5) {$2$};
            \node[] at (2.5,1.5) {$3$};
            \node[] at (3.5,1.5) {$4$};

            \node[] at (1.5,2.5) {$1$};
            \node[] at (2.5,2.5) {$3$};
            \node[] at (3.5,2.5) {$4$};

            \node[] at (2,-1) {$\frac{1-t}{1-qt}x_1x_2x_3^2x_4^2$};
        \end{scope}
        
        \begin{scope}[xshift=-16cm, yshift=-6cm]
            \draw[thick] (0,0) grid (4,1);
            \draw[thick] (1,1) grid (4,3);

            \node[] at (0.5,0.5) {$1$};
            \node[] at (1.5,0.5) {$2$};
            \node[] at (2.5,0.5) {$3$};
            \node[] at (3.5,0.5) {$4$};

            \node[] at (1.5,1.5) {$2$};
            \node[] at (2.5,1.5) {$3$};
            \node[] at (3.5,1.5) {$1$};

            \node[] at (1.5,2.5) {$4$};
            \node[] at (2.5,2.5) {$3$};
            \node[] at (3.5,2.5) {$1$};

            \node[] at (2,-1) {$\frac{q(1-t)^2}{(1-qt)(1-q^2t)}x_1^2x_2x_3^2x_4$};
        \end{scope}
        \begin{scope}[xshift=-8cm, yshift=-6cm]
            \draw[thick] (0,0) grid (4,1);
            \draw[thick] (1,1) grid (4,3);

            \node[] at (0.5,0.5) {$1$};
            \node[] at (1.5,0.5) {$2$};
            \node[] at (2.5,0.5) {$3$};
            \node[] at (3.5,0.5) {$4$};

            \node[] at (1.5,1.5) {$2$};
            \node[] at (2.5,1.5) {$3$};
            \node[] at (3.5,1.5) {$1$};

            \node[] at (1.5,2.5) {$2$};
            \node[] at (2.5,2.5) {$4$};
            \node[] at (3.5,2.5) {$1$};

            \node[] at (2,-1) {$\frac{q(1-t)^2}{(1-qt)(1-q^2t)}x_1^2x_2^2x_3x_4$};
        \end{scope}
        \begin{scope}[xshift=0cm, yshift=-6cm]
            \draw[thick] (0,0) grid (4,1);
            \draw[thick] (1,1) grid (4,3);

            \node[] at (0.5,0.5) {$1$};
            \node[] at (1.5,0.5) {$2$};
            \node[] at (2.5,0.5) {$3$};
            \node[] at (3.5,0.5) {$4$};

            \node[] at (1.5,1.5) {$2$};
            \node[] at (2.5,1.5) {$3$};
            \node[] at (3.5,1.5) {$1$};

            \node[] at (1.5,2.5) {$2$};
            \node[] at (2.5,2.5) {$3$};
            \node[] at (3.5,2.5) {$4$};

            \node[] at (2,-1) {$\frac{q(1-t)^2}{(1-qt)(1-q^2t)}x_1x_2^2x_3^2x_4$};
        \end{scope}
        \begin{scope}[xshift=8cm, yshift=-6cm]
            \draw[thick] (0,0) grid (4,1);
            \draw[thick] (1,1) grid (4,3);

            \node[] at (0.5,0.5) {$1$};
            \node[] at (1.5,0.5) {$2$};
            \node[] at (2.5,0.5) {$3$};
            \node[] at (3.5,0.5) {$4$};

            \node[] at (1.5,1.5) {$2$};
            \node[] at (2.5,1.5) {$3$};
            \node[] at (3.5,1.5) {$1$};

            \node[] at (1.5,2.5) {$2$};
            \node[] at (2.5,2.5) {$3$};
            \node[] at (3.5,2.5) {$1$};

            \node[] at (2,-1) {$\frac{1-t}{1-q^2t}x_1^2x_2^2x_3^2$};
        \end{scope}
        
        \begin{scope}[xshift=-16cm, yshift=-12cm]
            \draw[thick] (0,0) grid (4,1);
            \draw[thick] (1,1) grid (4,3);

            \node[] at (0.5,0.5) {$1$};
            \node[] at (1.5,0.5) {$2$};
            \node[] at (2.5,0.5) {$3$};
            \node[] at (3.5,0.5) {$4$};

            \node[] at (1.5,1.5) {$2$};
            \node[] at (2.5,1.5) {$1$};
            \node[] at (3.5,1.5) {$4$};

            \node[] at (1.5,2.5) {$3$};
            \node[] at (2.5,2.5) {$1$};
            \node[] at (3.5,2.5) {$4$};

            \node[] at (2,-1) {$\frac{q(1-t)^2}{(1-qt)(1-q^2t)}x_1^2x_2x_3x_4^2$};
        \end{scope}
        \begin{scope}[xshift=-8cm, yshift=-12cm]
            \draw[thick] (0,0) grid (4,1);
            \draw[thick] (1,1) grid (4,3);

            \node[] at (0.5,0.5) {$1$};
            \node[] at (1.5,0.5) {$2$};
            \node[] at (2.5,0.5) {$3$};
            \node[] at (3.5,0.5) {$4$};

            \node[] at (1.5,1.5) {$2$};
            \node[] at (2.5,1.5) {$1$};
            \node[] at (3.5,1.5) {$4$};

            \node[] at (1.5,2.5) {$2$};
            \node[] at (2.5,2.5) {$3$};
            \node[] at (3.5,2.5) {$4$};

            \node[] at (2,-1) {$\frac{q(1-t)^2}{(1-qt)(1-q^2t)}x_1x_2^2x_3x_4^2$};
        \end{scope}
        \begin{scope}[xshift=0cm, yshift=-12cm]
            \draw[thick] (0,0) grid (4,1);
            \draw[thick] (1,1) grid (4,3);

            \node[] at (0.5,0.5) {$1$};
            \node[] at (1.5,0.5) {$2$};
            \node[] at (2.5,0.5) {$3$};
            \node[] at (3.5,0.5) {$4$};

            \node[] at (1.5,1.5) {$2$};
            \node[] at (2.5,1.5) {$1$};
            \node[] at (3.5,1.5) {$4$};

            \node[] at (1.5,2.5) {$2$};
            \node[] at (2.5,2.5) {$1$};
            \node[] at (3.5,2.5) {$4$};

            \node[] at (2,-1) {$\frac{1-t}{1-q^2t}x_1^2x_2^2x_4^2$};
        \end{scope}
        \begin{scope}[xshift=8cm, yshift=-12cm]
            \draw[thick] (0,0) grid (4,1);
            \draw[thick] (1,1) grid (4,3);

            \node[] at (0.5,0.5) {$1$};
            \node[] at (1.5,0.5) {$2$};
            \node[] at (2.5,0.5) {$3$};
            \node[] at (3.5,0.5) {$4$};

            \node[] at (1.5,1.5) {$2$};
            \node[] at (2.5,1.5) {$1$};
            \node[] at (3.5,1.5) {$4$};

            \node[] at (1.5,2.5) {$2$};
            \node[] at (2.5,2.5) {$1$};
            \node[] at (3.5,2.5) {$3$};

            \node[] at (2,-1) {$\frac{(1-t)^2}{(1-qt)(1-q^2t)}x_1^2x_2^2x_3x_4$};
        \end{scope}
        
        \begin{scope}[xshift=-16cm, yshift=-18cm]
            \draw[thick] (0,0) grid (4,1);
            \draw[thick] (1,1) grid (4,3);

            \node[] at (0.5,0.5) {$1$};
            \node[] at (1.5,0.5) {$2$};
            \node[] at (2.5,0.5) {$3$};
            \node[] at (3.5,0.5) {$4$};

            \node[] at (1.5,1.5) {$1$};
            \node[] at (2.5,1.5) {$3$};
            \node[] at (3.5,1.5) {$4$};

            \node[] at (1.5,2.5) {$1$};
            \node[] at (2.5,2.5) {$3$};
            \node[] at (3.5,2.5) {$2$};

            \node[] at (2,-1) {$\frac{q(1-t)^2}{(1-qt)(1-q^2t)}x_1x_2x_3^2x_4^2$};
        \end{scope}
        \begin{scope}[xshift=-8cm, yshift=-18cm]
            \draw[thick] (0,0) grid (4,1);
            \draw[thick] (1,1) grid (4,3);

            \node[] at (0.5,0.5) {$1$};
            \node[] at (1.5,0.5) {$2$};
            \node[] at (2.5,0.5) {$3$};
            \node[] at (3.5,0.5) {$4$};

            \node[] at (1.5,1.5) {$1$};
            \node[] at (2.5,1.5) {$3$};
            \node[] at (3.5,1.5) {$4$};

            \node[] at (1.5,2.5) {$1$};
            \node[] at (2.5,2.5) {$3$};
            \node[] at (3.5,2.5) {$4$};

            \node[] at (2,-1) {$\frac{1-t}{1-q^2t}x_1^2x_3^2x_4^2$};
        \end{scope}
        \begin{scope}[xshift=0cm, yshift=-18cm]
            \draw[thick] (0,0) grid (4,1);
            \draw[thick] (1,1) grid (4,3);

            \node[] at (0.5,0.5) {$1$};
            \node[] at (1.5,0.5) {$2$};
            \node[] at (2.5,0.5) {$3$};
            \node[] at (3.5,0.5) {$4$};

            \node[] at (1.5,1.5) {$1$};
            \node[] at (2.5,1.5) {$3$};
            \node[] at (3.5,1.5) {$4$};

            \node[] at (1.5,2.5) {$1$};
            \node[] at (2.5,2.5) {$3$};
            \node[] at (3.5,2.5) {$2$};

            \node[] at (2,-1) {$\frac{(1-t)^2}{(1-qt)(1-q^2t)}x_1^2x_2x_3^2x_4$};
        \end{scope}
        \begin{scope}[xshift=8cm, yshift=-18cm]
            \draw[thick] (0,0) grid (4,1);
            \draw[thick] (1,1) grid (4,3);

            \node[] at (0.5,0.5) {$1$};
            \node[] at (1.5,0.5) {$2$};
            \node[] at (2.5,0.5) {$3$};
            \node[] at (3.5,0.5) {$4$};

            \node[] at (1.5,1.5) {$1$};
            \node[] at (2.5,1.5) {$3$};
            \node[] at (3.5,1.5) {$4$};

            \node[] at (1.5,2.5) {$1$};
            \node[] at (2.5,2.5) {$2$};
            \node[] at (3.5,2.5) {$4$};

            \node[] at (2,-1) {$\frac{(1-t)^2}{(1-qt)(1-q^2t)}x_1^2x_2x_3x_4^2$};
        \end{scope}
    \end{tikzpicture}
    \caption{All restricted nonattacking fillings for the composition $\mu = (0,2,2,2)$ with their mixed weights.}
    \label{fig:RNF(0,2,2,2)}
\end{figure}

In general, the mixed $q,t$-weight of a restricted nonattacking filling encodes the sum of the $q,t$-weights of several nonattacking fillings. In this way, we view restricted nonattacking fillings as a compression of nonattacking fillings, see Figure~\ref{fig:CompressionExample} for an example. 

\begin{figure}[ht]
    \centering
    \begin{tikzpicture}[scale=0.5]
        \begin{scope}[xshift=-18cm, yshift=0cm]
            \draw[thick] (0,0) grid (4,1);
            \draw[thick] (1,1) grid (4,3);

            \node[] at (0.5,0.5) {$1$};
            \node[] at (1.5,0.5) {$2$};
            \node[] at (2.5,0.5) {$3$};
            \node[] at (3.5,0.5) {$4$};

            \node[] at (1.5,1.5) {$2$};
            \node[] at (2.5,1.5) {$1$};
            \node[] at (3.5,1.5) {$4$};

            \node[] at (1.5,2.5) {$2$};
            \node[] at (2.5,2.5) {$1$};
            \node[] at (3.5,2.5) {$3$};

            \node[] at (2,-1) {$\frac{t(1-t)^2}{(1-qt^3)(1-q^2t^2)}$};
        \end{scope}
        \begin{scope}[xshift=-10cm, yshift=0cm]
            \draw[thick] (0,0) grid (4,1);
            \draw[thick] (1,1) grid (4,3);

            \node[] at (0.5,0.5) {$1$};
            \node[] at (1.5,0.5) {$2$};
            \node[] at (2.5,0.5) {$3$};
            \node[] at (3.5,0.5) {$4$};

            \node[] at (1.5,1.5) {$2$};
            \node[] at (2.5,1.5) {$1$};
            \node[] at (3.5,1.5) {$4$};

            \node[] at (1.5,2.5) {$2$};
            \node[] at (2.5,2.5) {$3$};
            \node[] at (3.5,2.5) {$1$};

            \node[] at (2,-1) {$\frac{qt^3(1-t)^3}{(1-qt^2)(1-qt^3)
            (1-q^2t^2)}$};
        \end{scope}
        \begin{scope}[xshift=0cm, yshift=0cm]
            \draw[thick] (0,0) grid (4,1);
            \draw[thick] (1,1) grid (4,3);

            \node[] at (0.5,0.5) {$1$};
            \node[] at (1.5,0.5) {$2$};
            \node[] at (2.5,0.5) {$3$};
            \node[] at (3.5,0.5) {$4$};

            \node[] at (1.5,1.5) {$2$};
            \node[] at (2.5,1.5) {$1$};
            \node[] at (3.5,1.5) {$4$};

            \node[] at (1.5,2.5) {$3$};
            \node[] at (2.5,2.5) {$2$};
            \node[] at (3.5,2.5) {$1$};

            \node[] at (2,-1) {$\frac{q^2t^4(1-t)^4}{(1-qt)(1-qt^2)(1-qt^3)
            (1-q^2t^2)}$};
        \end{scope}
        \begin{scope}[xshift=10cm, yshift=0cm]
            \draw[thick] (0,0) grid (4,1);
            \draw[thick] (1,1) grid (4,3);

            \node[] at (0.5,0.5) {$1$};
            \node[] at (1.5,0.5) {$2$};
            \node[] at (2.5,0.5) {$3$};
            \node[] at (3.5,0.5) {$4$};

            \node[] at (1.5,1.5) {$2$};
            \node[] at (2.5,1.5) {$1$};
            \node[] at (3.5,1.5) {$4$};

            \node[] at (1.5,2.5) {$3$};
            \node[] at (2.5,2.5) {$1$};
            \node[] at (3.5,2.5) {$2$};

            \node[] at (2,-1) {$\frac{qt^2(1-t)^3}{(1-qt)(1-qt^3)(1-q^2t^2)}$};
        \end{scope}
        \begin{scope}[xshift=-18cm, yshift=-6cm]
            \draw[thick] (0,0) grid (4,1);
            \draw[thick] (1,1) grid (4,3);

            \node[] at (0.5,0.5) {$1$};
            \node[] at (1.5,0.5) {$2$};
            \node[] at (2.5,0.5) {$3$};
            \node[] at (3.5,0.5) {$4$};

            \node[] at (1.5,1.5) {$1$};
            \node[] at (2.5,1.5) {$2$};
            \node[] at (3.5,1.5) {$4$};

            \node[] at (1.5,2.5) {$1$};
            \node[] at (2.5,2.5) {$2$};
            \node[] at (3.5,2.5) {$3$};

            \node[] at (2,-1) {$\frac{(1-t)^3}{(1-qt^3)(1-q^2t)(1-q^2t^2)}$};
        \end{scope}
        \begin{scope}[xshift=-10cm, yshift=-6cm]
            \draw[thick] (0,0) grid (4,1);
            \draw[thick] (1,1) grid (4,3);

            \node[] at (0.5,0.5) {$1$};
            \node[] at (1.5,0.5) {$2$};
            \node[] at (2.5,0.5) {$3$};
            \node[] at (3.5,0.5) {$4$};

            \node[] at (1.5,1.5) {$1$};
            \node[] at (2.5,1.5) {$2$};
            \node[] at (3.5,1.5) {$4$};

            \node[] at (1.5,2.5) {$1$};
            \node[] at (2.5,2.5) {$3$};
            \node[] at (3.5,2.5) {$2$};

            \node[] at (2,-1) {$\frac{qt^2(1-t)^4}{(1-qt^3)(1-qt^2)(1-q^2t)(1-q^2t^2)}$};
        \end{scope}
        \begin{scope}[xshift=0cm, yshift=-6cm]
            \draw[thick] (0,0) grid (4,1);
            \draw[thick] (1,1) grid (4,3);

            \node[] at (0.5,0.5) {$1$};
            \node[] at (1.5,0.5) {$2$};
            \node[] at (2.5,0.5) {$3$};
            \node[] at (3.5,0.5) {$4$};

            \node[] at (1.5,1.5) {$1$};
            \node[] at (2.5,1.5) {$2$};
            \node[] at (3.5,1.5) {$4$};

            \node[] at (1.5,2.5) {$3$};
            \node[] at (2.5,2.5) {$1$};
            \node[] at (3.5,2.5) {$2$};

            \node[] at (2,-1) {$\frac{qt(1-t)^5}{(1-qt)(1-qt^2)(1-qt^3)(1-q^2t)(1-q^2t^2)}$};
        \end{scope}
        \begin{scope}[xshift=10cm, yshift=-6cm]
            \draw[thick] (0,0) grid (4,1);
            \draw[thick] (1,1) grid (4,3);

            \node[] at (0.5,0.5) {$1$};
            \node[] at (1.5,0.5) {$2$};
            \node[] at (2.5,0.5) {$3$};
            \node[] at (3.5,0.5) {$4$};

            \node[] at (1.5,1.5) {$1$};
            \node[] at (2.5,1.5) {$2$};
            \node[] at (3.5,1.5) {$4$};

            \node[] at (1.5,2.5) {$3$};
            \node[] at (2.5,2.5) {$2$};
            \node[] at (3.5,2.5) {$1$};

            \node[] at (2,-1) {$\frac{qt^2(1-t)^4}{(1-qt)(1-qt^3)(1-q^2t)(1-q^2t^2)}$};
        \end{scope}
        \draw [decorate,decoration={brace,amplitude=10pt,mirror,raise=4ex}]
            (-18,-7.5) -- (14,-7.5) ;
        \begin{scope}[xshift = -4.5cm, yshift=-13cm]
            \node[] at (-4, 2) {Restricted NAF:};
        
            \draw[thick] (0,0) grid (4,1);
            \draw[thick] (1,1) grid (4,3);

            \node[] at (0.5,0.5) {$1$};
            \node[] at (1.5,0.5) {$2$};
            \node[] at (2.5,0.5) {$3$};
            \node[] at (3.5,0.5) {$4$};

            \node[] at (1.5,1.5) {$2$};
            \node[] at (2.5,1.5) {$1$};
            \node[] at (3.5,1.5) {$4$};

            \node[] at (1.5,2.5) {$2$};
            \node[] at (2.5,2.5) {$1$};
            \node[] at (3.5,2.5) {$3$};

            \node[] at (2,-1) {$\wt_m = \frac{(1-t)^2}{(1-qt)(1-q^2t)}$};
        \end{scope}
    \end{tikzpicture}
    \caption{The set of nonattacking fillings whose $q,t$-weights sum to the 
    $q,t$-mixed weight of the restricted nonattacking filling at the bottom. 
    }
    \label{fig:CompressionExample}
\end{figure}
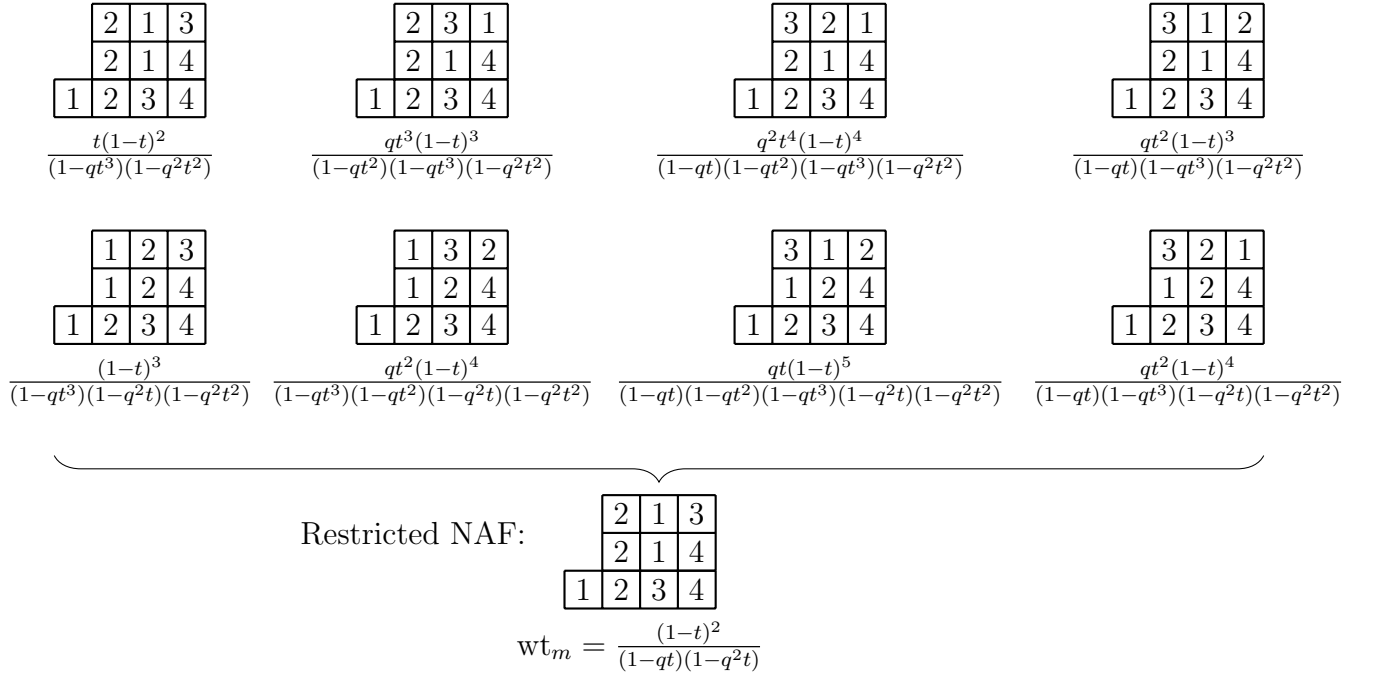

The remainder of the paper is dedicated to proving Theorem~\ref{thm:MainResult}. To do so, we define operators that act on mixed fillings. The action behaves well with respect to the coinversions and the major index statistics. In some cases, we can use this action to show that the mixed polynomial of one set of mixed fillings is equal to the polynomial of another set of mixed fillings. Starting from the usual nonattacking fillings in the \HHL model, we use the operators to get to the restricted nonattacking fillings for the compositions in Theorem~\ref{thm:MainResult} by inducting over the set of boxes in $\dg'(\mu)$. However, this method does not extend easily to the general case because the set of boxes that hit a box is larger, and it is not clear how to use the operators to describe the exact set of nonattacking fillings that sum to a single restricted nonattacking filling, see Section~\ref{sect:Conclusion} for more info and an example.

\section{Operator on Columns} \label{section:Operator}

In this section, we define a set of operators that act on nonattacking fillings. Then, we prove that major index and coinversions change in a predictable manner. Finally, we describe the effect of the operators when applied to specific mixed nonattacking fillings. 

Suppose $\mu = (\mu_1, \mu_2, \ldots, \mu_n)$ is a composition such that there exist columns $\mu_i = \mu_{i'}$, where $i < i'$. Furthermore, suppose there exists a $1 \le j \le \mu_i$ such that $\mu_c \le j-2$ for all $i < c < i'$. As a result, the boxes $(i,j)$ and $(i',j)$ are consecutive in the diagram of $\mu$. In Figure~\ref{fig:OperatorAction}, we show a generic diagram for this setup.

\begin{definition}
We define the action of the $\T{i}{j}$ on a filling  $|ws$ algorithmically. Start by swapping entries $\ws(i,j)$ and $\ws(i,j)$. Then, if either of the following conditions is met, the operator swaps the entries in boxes $(i,j+1)$ and $(i',j+1)$:
\begin{itemize}
    \item $\ws(i,j) = \ws(i,j+1)$,
    \item $\ws(i',j) = \ws(i',j+1)$, or 
    \item the orientation of the triple $\{\ws(i,j+1),\ws(i',j+1), \ws(i',j)\}$ is different from the orientation of the triple $\{\ws(i,j+1),\ws(i',j+1), \ws(i,j)\}$ in $\ws$.
\end{itemize}
We proceed to swap the entries in boxes $(i,j+2)$ and $(i',j+2)$ using the same conditions, but boxes $(i,j)$ and $(i',j)$ are replaced with $(i,j+1)$ and $(i',j+1)$ respectively when deciding to swap these entries. The operator continues in this way until either there are no more boxes to swap or a swap is not necessary. The resulting filling is a nonattacking filling $\ws'$, and we say $\T{i}{j}(\ws) = \ws'$. If either $\ws(i',j)=\ws(i',j-1)$ or, at any point during the swapping procedure, a trivial pair is formed in column $i'$, the operator acts trivially, that is, $\T{i}{j}(\ws) = \ws$. We write $\ws'$ for $\T{i}{j}(\ws)$ whenever the context is clear to simplify notation.
\end{definition}

In Figure~\ref{fig:OperatorAction}, we show the most general case of $\T{i}{j}$ acting on a filling $\ws$. The leftmost nonattacking filling is $\ws$ and the rightmost nonattacking filling is $\T{i}{j}(\ws)$. The letters in the boxes denote the entries of the boxes and not the boxes themselves. We assume that the last entries to swap are $c_k$ and $d_k$; thus, the orientation of the type II triple $\{a,b,d_k\}$ in $\ws$ is the same as the orientation of the type II triple $\{a,b,c_k\}$ in $\ws'$.
\begin{figure}[h!]
    \centering
    \begin{tikzpicture}
    \begin{scope}[xshift = -4cm, yshift = 0cm]
        \draw[thick] (-.1,0) -- (-.1,5.5);
        \draw[thick] (0.6,0) -- (0.6,5.5);
        
        \draw[thick] (1.9,0) -- (1.9,5.5);
        \draw[thick] (2.6,0) -- (2.6,5.5);

        \draw[thick] (-0.1,0.5) -- (2.6,0.5);
        
        \draw[thick] (-.1,1) -- (0.6,1);
        \draw[thick] (-.1,1.5) -- (0.6,1.5);
        \draw[thick] (-.1,2) -- (0.6,2);
        \draw[thick] (-.1,2.5) -- (0.6,2.5);
        \draw[thick] (-.1,3) -- (0.6,3);
        \draw[thick] (-.1,3.5) -- (0.6,3.5);
        \draw[thick] (-.1,4) -- (0.6,4);
        \draw[thick] (-.1,4.5) -- (0.6,4.5);
        \draw[thick] (-.1,5) -- (0.6,5);

        \node at (0.25,0.25) {$\vdots$};
        \node at (0.25,0.75) {$u$};
        \node at (0.25,1.25) {$x$};
        \node[blue box] at (0.25,1.75) {$c_1$};
        \node at (0.25,2.35) {$\vdots$};
        \node at (0.25,2.75) {$c_{i-1}$};
        \node at (0.25,3.25) {$c_i$};
        \node at (0.25,3.85) {$\vdots$};
        \node at (0.25,4.25) {$c_k$};
        \node at (0.25,4.75) {$a$};
        \node at (0.25,5.35) {$\vdots$};

        \node at (1.25, 0.25) {$\ldots$};

        \draw[thick] (1.9,1) -- (2.6,1);
        \draw[thick] (1.9,1.5) -- (2.6,1.5);
        \draw[thick] (1.9,2) -- (2.6,2);
        \draw[thick] (1.9,2.5) -- (2.6,2.5);
        \draw[thick] (1.9,3) -- (2.6,3);
        \draw[thick] (1.9,3.5) -- (2.6,3.5);
        \draw[thick] (1.9,4) -- (2.6,4);
        \draw[thick] (1.9,4.5) -- (2.6,4.5);
        \draw[thick] (1.9,5) -- (2.6,5);

        \node at (2.25,0.25) {$\vdots$};
        \node at (2.25,0.75) {$w$};
        \node[blue box] at (2.25,1.25) {$v$};
        \node[blue box] at (2.25,1.75) {$d_1$};
        \node at (2.25,2.35) {$\vdots$};
        \node at (2.25,2.75) {$d_{i-1}$};
        \node at (2.25,3.25) {$d_i$};
        \node at (2.25,3.85) {$\vdots$};
        \node at (2.25,4.25) {$d_k$};
        \node at (2.25,4.75) {$b$};
        \node at (2.25,5.35) {$\vdots$};
        
    \end{scope}
    \draw[thick, ->] (-1.25,2.65) -- (-0.25,2.65);
    \begin{scope}[xshift = 0cm, yshift = 0cm]
        \draw[thick] (-.1,0) -- (-.1,5.5);
        \draw[thick] (0.6,0) -- (0.6,5.5);
        
        \draw[thick] (1.9,0) -- (1.9,5.5);
        \draw[thick] (2.6,0) -- (2.6,5.5);

        \draw[thick] (-0.1,0.5) -- (2.6,0.5);
        
        \draw[thick] (-.1,1) -- (0.6,1);
        \draw[thick] (-.1,1.5) -- (0.6,1.5);
        \draw[thick] (-.1,2) -- (0.6,2);
        \draw[thick] (-.1,2.5) -- (0.6,2.5);
        \draw[thick] (-.1,3) -- (0.6,3);
        \draw[thick] (-.1,3.5) -- (0.6,3.5);
        \draw[thick] (-.1,4) -- (0.6,4);
        \draw[thick] (-.1,4.5) -- (0.6,4.5);
        \draw[thick] (-.1,5) -- (0.6,5);

        \node[] at (0.25,0.25) {$\vdots$};
        \node at (0.25,0.75) {$u$};
        \node at (0.25,1.25) {$v$};
        \node[blue box] at (0.25,1.75) {$c_1$};
        \node at (0.25,2.35) {$\vdots$};
        \node at (0.25,2.75) {$c_{i-1}$};
        \node at (0.25,3.25) {$c_i$};
        \node at (0.25,3.85) {$\vdots$};
        \node at (0.25,4.25) {$c_k$};
        \node at (0.25,4.75) {$a$};
        \node at (0.25,5.35) {$\vdots$};

        \node at (1.25, 0.25) {$\ldots$};

        \draw[thick] (1.9,1) -- (2.6,1);
        \draw[thick] (1.9,1.5) -- (2.6,1.5);
        \draw[thick] (1.9,2) -- (2.6,2);
        \draw[thick] (1.9,2.5) -- (2.6,2.5);
        \draw[thick] (1.9,3) -- (2.6,3);
        \draw[thick] (1.9,3.5) -- (2.6,3.5);
        \draw[thick] (1.9,4) -- (2.6,4);
        \draw[thick] (1.9,4.5) -- (2.6,4.5);
        \draw[thick] (1.9,5) -- (2.6,5);

        \node at (2.25,0.25) {$\vdots$};
        \node at (2.25,0.75) {$w$};
        \node[blue box] at (2.25,1.25) {$x$};
        \node[blue box] at (2.25,1.75) {$d_1$};
        \node at (2.25,2.35) {$\vdots$};
        \node at (2.25,2.75) {$d_{i-1}$};
        \node at (2.25,3.25) {$d_i$};
        \node at (2.25,3.85) {$\vdots$};
        \node at (2.25,4.25) {$d_k$};
        \node at (2.25,4.75) {$b$};
        \node at (2.25,5.35) {$\vdots$};
        
    \end{scope}
    \node at (3.75,2.65) {$\scalebox{1.5}{\dots}$};
    \begin{scope}[xshift = 5cm, yshift = 0cm]
        \draw[thick] (-.1,0) -- (-.1,5.5);
        \draw[thick] (0.6,0) -- (0.6,5.5);
        
        \draw[thick] (1.9,0) -- (1.9,5.5);
        \draw[thick] (2.6,0) -- (2.6,5.5);

        \draw[thick] (-0.1,0.5) -- (2.6,0.5);
        
        \draw[thick] (-.1,1) -- (0.6,1);
        \draw[thick] (-.1,1.5) -- (0.6,1.5);
        \draw[thick] (-.1,2) -- (0.6,2);
        \draw[thick] (-.1,2.5) -- (0.6,2.5);
        \draw[thick] (-.1,3) -- (0.6,3);
        \draw[thick] (-.1,3.5) -- (0.6,3.5);
        \draw[thick] (-.1,4) -- (0.6,4);
        \draw[thick] (-.1,4.5) -- (0.6,4.5);
        \draw[thick] (-.1,5) -- (0.6,5);

        \node at (0.25,0.25) {$\vdots$};
        \node at (0.25,0.75) {$u$};
        \node at (0.25,1.25) {$v$};
        \node at (0.25,1.75) {$d_1$};
        \node at (0.25,2.35) {$\vdots$};
        \node at (0.25,2.75) {$d_{i-1}$};
        \node at (0.25,3.25) {$d_i$};
        \node at (0.25,3.85) {$\vdots$};
        \node[blue box] at (0.25,4.25) {$c_k$};
        \node[] at (0.25,4.75) {$a$};
        \node at (0.25,5.35) {$\vdots$};

        \node at (1.25, 0.25) {$\ldots$};

        \draw[thick] (1.9,1) -- (2.6,1);
        \draw[thick] (1.9,1.5) -- (2.6,1.5);
        \draw[thick] (1.9,2) -- (2.6,2);
        \draw[thick] (1.9,2.5) -- (2.6,2.5);
        \draw[thick] (1.9,3) -- (2.6,3);
        \draw[thick] (1.9,3.5) -- (2.6,3.5);
        \draw[thick] (1.9,4) -- (2.6,4);
        \draw[thick] (1.9,4.5) -- (2.6,4.5);
        \draw[thick] (1.9,5) -- (2.6,5);

        \node at (2.25,0.25) {$\vdots$};
        \node at (2.25,0.75) {$w$};
        \node at (2.25,1.25) {$x$};
        \node at (2.25,1.75) {$c_1$};
        \node at (2.25,2.35) {$\vdots$};
        \node at (2.25,2.75) {$c_{i-1}$};
        \node at (2.25,3.25) {$c_i$};
        \node at (2.25,3.85) {$\vdots$};
        \node[blue box] at (2.25,4.25) {$d_k$};
        \node[] at (2.25,4.75) {$b$};
        \node at (2.25,5.35) {$\vdots$};
        
    \end{scope}
    \draw[thick, ->] (7.75,2.65) -- (8.75,2.65);
    \begin{scope}[xshift = 9cm, yshift = 0cm]
        \draw[thick] (-.1,0) -- (-.1,5.5);
        \draw[thick] (0.6,0) -- (0.6,5.5);
        
        \draw[thick] (1.9,0) -- (1.9,5.5);
        \draw[thick] (2.6,0) -- (2.6,5.5);

        \draw[thick] (-0.1,0.5) -- (2.6,0.5);
        
        \draw[thick] (-.1,1) -- (0.6,1);
        \draw[thick] (-.1,1.5) -- (0.6,1.5);
        \draw[thick] (-.1,2) -- (0.6,2);
        \draw[thick] (-.1,2.5) -- (0.6,2.5);
        \draw[thick] (-.1,3) -- (0.6,3);
        \draw[thick] (-.1,3.5) -- (0.6,3.5);
        \draw[thick] (-.1,4) -- (0.6,4);
        \draw[thick] (-.1,4.5) -- (0.6,4.5);
        \draw[thick] (-.1,5) -- (0.6,5);

        \node at (0.25,0.25) {$\vdots$};
        \node at (0.25,0.75) {$u$};
        \node at (0.25,1.25) {$v$};
        \node at (0.25,1.75) {$d_1$};
        \node at (0.25,2.35) {$\vdots$};
        \node at (0.25,2.75) {$d_{i-1}$};
        \node at (0.25,3.25) {$d_i$};
        \node at (0.25,3.85) {$\vdots$};
        \node at (0.25,4.25) {$d_k$};
        \node[blue box] at (0.25,4.75) {$a$};
        \node at (0.25,5.35) {$\vdots$};

        \node at (1.25, 0.25) {$\ldots$};

        \draw[thick] (1.9,1) -- (2.6,1);
        \draw[thick] (1.9,1.5) -- (2.6,1.5);
        \draw[thick] (1.9,2) -- (2.6,2);
        \draw[thick] (1.9,2.5) -- (2.6,2.5);
        \draw[thick] (1.9,3) -- (2.6,3);
        \draw[thick] (1.9,3.5) -- (2.6,3.5);
        \draw[thick] (1.9,4) -- (2.6,4);
        \draw[thick] (1.9,4.5) -- (2.6,4.5);
        \draw[thick] (1.9,5) -- (2.6,5);

        \node at (2.25,0.25) {$\vdots$};
        \node at (2.25,0.75) {$w$};
        \node at (2.25,1.25) {$x$};
        \node at (2.25,1.75) {$c_1$};
        \node at (2.25,2.35) {$\vdots$};
        \node at (2.25,2.75) {$c_{i-1}$};
        \node at (2.25,3.25) {$c_i$};
        \node at (2.25,3.85) {$\vdots$};
        \node[blue box] at (2.25,4.25) {$c_k$};
        \node[blue box] at (2.25,4.75) {$b$};
        \node at (2.25,5.35) {$\vdots$};
        
    \end{scope}
    \end{tikzpicture}.
    \caption{The operator $\mathcal{T}_{(i,j)}$ applied to 
    columns of the same height.}
    \label{fig:OperatorAction}
\end{figure}
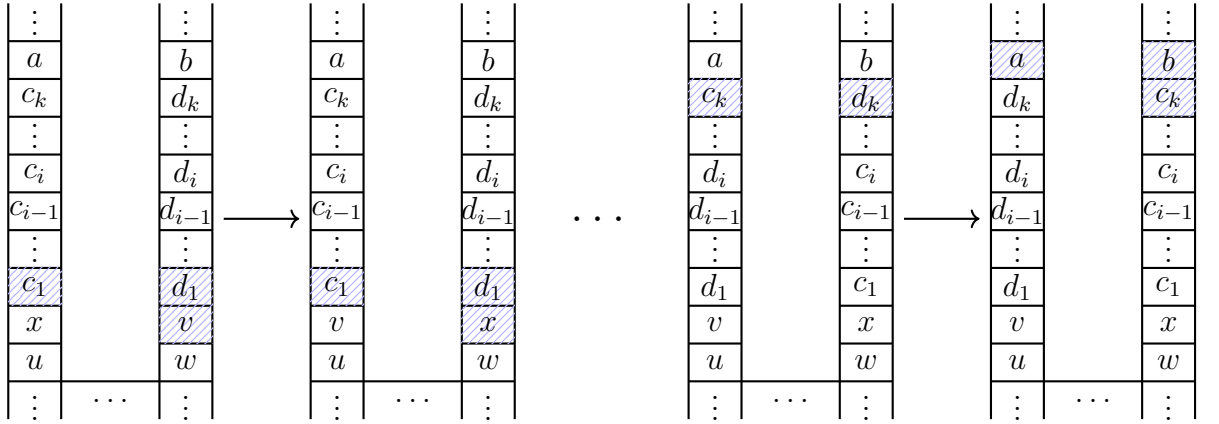

In the remaining part of this section, we analyze the change in coinversions and major index from $\ws$ to $\ws'$. We define
$$\Delta \maj = \maj(\ws') - \maj(\ws) \qquad \text{ and } \qquad \Delta \coinv = \coinv(\ws') - \coinv(\ws).$$
To calculate $\Delta \maj$ and $\Delta \coinv$, we compute the contribution to the change in major index and the change in coinversions row by row. We start with the topmost row containing entries $a$ and $b$ in Figure~\ref{fig:OperatorAction}.

\begin{lemma} \label{lem:OperatorTop}
Consider the setup presented in Figure~\ref{fig:OperatorAction}. The topmost row, containing entries $a$ and $b$, does not contribute to the change in coinversions nor the change in major index. 
\end{lemma}

\begin{proof}
    We begin by showing that the number of type I and type II coinversion triples does not change. Then, we finish by proving that the number of descents does not change; hence, the major index does not change. 

    First, notice that the only place where coinversion triples can change is in the ones that involve $a$ and $c_k$ or $b$ and $d_k$. We include the setup for this case in Figure~\ref{fig:TopRowSetup}. All that we have to do is show that the number of coinversion triples of the form $\{f, a, c_k\}$ and $\{f, b, d_k\}$ in $\ws$ equals the number of coinversion triples of the form $\{f, a, d_k\}$ and $\{f, b, c_k\}$ in $\ws'$. 

    \begin{figure}[h!]
        \centering
        \begin{tikzpicture}[scale = 0.5]
        \begin{scope}[]
            \draw[] (0,0) grid (2,2);
            \draw[] (-2,1) grid (-1,2);
            \draw[] (3,0) grid (4,1);

            \node[] at (0.5,0.5) {$c_k$};
            \node[] at (1.5,0.5) {$d_k$};
            \node[] at (0.5,1.5) {$a$};
            \node[] at (1.5,1.5) {$b$};
            \node[] at (-1.5,1.5) {$f$};
            \node[] at (3.5,0.5) {$f'$};

            \node at (1,3) {$\underline{\ws}$};
        \end{scope}
        \begin{scope}[xshift = 10cm]
            \draw[] (0,0) grid (2,2);
            \draw[] (-2,1) grid (-1,2);
            \draw[] (3,0) grid (4,1);

            \node[] at (0.5,0.5) {$d_k$};
            \node[] at (1.5,0.5) {$c_k$};
            \node[] at (0.5,1.5) {$a$};
            \node[] at (1.5,1.5) {$b$};
            \node[] at (-1.5,1.5) {$f$};
            \node[] at (3.5,0.5) {$f'$};

            \node at (1,3) {$\underline{\ws'}$};
        \end{scope}
        \end{tikzpicture}
        \caption{The form of the coinversion triples that we consider in this proof.}
        \label{fig:TopRowSetup}
    \end{figure}
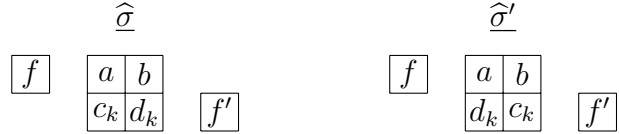
    
    By definition, the type II triple $\{a,b,d_k\}$ in $\ws$ and the type II triple $\{a,b,c_k\}$ in $\ws'$ have the same orientation. Thus, $\{a,b,d_k\}$ is a type II coinversion triple in $\ws$ if and only if $\{a,b,c_k\}$ is a type II coinversion triple in $\ws'$. This means that we only need to count the number of coinversions where $f$ is not $a$. 
    
    Without loss of generality, assume that $d_k < c_k$.  In Table~\ref{tab:OrientationTableCC}, we show the possible orientations for each set of type II triples. The first column is the direction in which the entries increase. The second and third columns tell us whether the triple listed at the top is a type II coinversion triple in $\ws$. Finally, the fourth and fifth columns tell us whether the triples at the top of the column are type II coinversion triples in $\ws'$. The table confirms that the number of counterclockwise and clockwise pairs involving the triples at the top of the columns stays the same. Thus, the number of type II coinversion triples of the form $\{f,a,c_k\}$ and $\{f,a,d_k\}$ in $\ws$ equals the number of type II coinversion triples of the form $\{f,a,d_k\}$ and $\{f,b,c_k\}$ in $\ws'$. 

    Similar reasoning shows that this is also true for type I coinversion triples of the form $\{f,a,c_k\}$ and $\{f,a,d_k\}$ in $\ws$ and $\{f,a,d_k\}$ and $\{f,a,c_k\}$ in $\ws'$. The table for type I coinversion triples can be derived from Table~\ref{tab:OrientationTableCC} by swapping counterclockwise and clockwise. 
    \begin{table}[h!]
        \centering
\newcommand{\cycle}[5]{
    \begin{tikzpicture}[baseline=(current bounding box.center), scale=0.7, 
                        every node/.style={font=\sffamily\small, inner sep=2pt}]
        \def\R{0.8cm} 
        \coordinate (c1) at (90:\R);
        \coordinate (c2) at (18:\R);
        \coordinate (c3) at (-54:\R);
        \coordinate (c4) at (234:\R);
        \coordinate (c5) at (162:\R);

        \node (n1) at (c1) {#1};
        \node (n2) at (c2) {#2};
        \node (n3) at (c3) {#3};
        \node (n4) at (c4) {#4};
        \node (n5) at (c5) {#5};

        \draw[thick, <-] (n1) -- (n2);
        \draw[thick, <-] (n2) -- (n3);
        \draw[thick, <-] (n3) -- (n4);
        \draw[thick, <-] (n4) -- (n5);
        \draw[thick, <-] (n5) -- (n1);
    \end{tikzpicture}
}

\renewcommand{\arraystretch}{2} 

    \begin{tabular}{c|c|c|c|c}
     counterclockwise & $f,a,c_k$ & $f,b,d_k$ & $f,a,d_k$ & $f,b,c_k$ \\
    \hline
    \cycle{$a$}{$b$}{$c_k$}{$d_k$}{$f$} & counter & counter & counter & counter \\
    \cycle{$a$}{$b$}{$c_k$}{$f$}{$d_k$} & counter & clock & clock & counter \\
    \cycle{$a$}{$b$}{$f$}{$c_k$}{$d_k$} & clock & clock & clock & clock \\
    \cycle{$a$}{$f$}{$b$}{$c_k$}{$d_k$} & clock & counter & clock & counter \\
    \end{tabular}
        \caption{Possible relative orderings for entries $f,a,b,d_k,c_k$ (type II triples): columns 2 and 3 correspond to $\ws$ and columns 4 and 5 correspond to $\ws'$}
        \label{tab:OrientationTableCC}
    \end{table}

We have shown that the number of coinversion at the top does not change. We finish by showing that the number of descents stays the same. 

Suppose that $\{a,b,d_k\}$ is a type II coinversion triple in $\ws$. Then, the triple $\{a,b,c_k\}$ is also a type II coinversion triple. If $c_k < d_k$, then
$$c_k < d_k < b < a, \qquad a < c_k < d_k < b, \qquad c_k < b < a < d_k, \qquad and \qquad b < a < c_k < d_k$$ 
are the only possible orderings such that $\{a,b,c_k\}$ is also a type II coinversion triple. By looking at Figure~\ref{fig:TopRowSetup}, the total number of descents involving entries $a$ and $b$ is preserved. The same is true if $d_k < c_k$. Therefore, the major index at the top row is not affected by swapping $c_k$ and $d_k$. 
\end{proof}

Next we show that the middle swaps do not affect the change in major index, and that the change in coinversion is rarely affected. Specifically, we are considering the middle entries $c_i,c_{i-1},d_i$ and $d_{i-1}$ for $i = 1, \ldots k$, where $c_0 = x$ and $d_0 = v$. The following lemma summarizes all the possible ways that the operators can affect the major index and coinversions.

\begin{lemma} \label{lem:OperatorMiddle}
Consider the entries $c_i$, $c_{i-1}$, $d_i$, $d_{i-1}$ for $i = 1, 2, \ldots, k$, where $c_0 = x$ and $d_0 = v$, as presented in Figure~\ref{fig:OperatorAction}. Then row $i$ does not contribute to the change in the major index, and row $i$ contributes to the change in coinversions in one of five ways:
\begin{enumerate}
    \item If $c_i$, $c_{i-1}$, $d_i$, and $d_{i-1}$ are distinct, then row $i$ contributes no change in coinversions.
    \item If $c_i = c_{i-1}$ and $d_i = d_{i-1}$, then row $i$ contributes no change in the coinversions.
    \item If $c_i = c_{i-1}$, $d_i \neq d_{i-1}$, and the triple $\{c_i,d_i,d_{i-1}\}$ is a coinversion triple, then the number of coinversions decreases by $1$.
    \item If $d_i = d_{i-1}$, $c_i \neq c_{i-1}$, and the triple $\{d_i,c_i,c_{i-1}\}$ in $\ws'$ is a coinversion triple, then the number of coinversions increases by $1$.
    \item If $c_i = c_{i-1}$ and $d_i = d_{i-1}$, then row $i$ contributes no change to the coinversions.
\end{enumerate}

\end{lemma}

\begin{proof}
We begin with the major index. Notice that $c_i$ is above $c_{i-1}$ and $d_i$ is above $d_{i-1}$ in both fillings; hence, if a descent occurred in $\ws$ between either box, then that same descent is also in $\ws'$. Since both columns have the same height and the number of descents is the same between $\ws$ and $\ws'$ at row $i$. The change contributed by row $i$ is zero.

We continue by proving each part separately. 
\begin{enumerate}
\item Suppose that $c_i$, $c_{i-1}$, $d_i$, and $d_{i-1}$ are distinct. 
Then the coinversion triples of the form $\{f,c_i,c_{i-1}\}$ and $\{f,d_i,d_{i-1}\}$ ($f \neq c_i$) in $\ws$ also occur in the exact same way in $\ws'$. Hence we only need to verify that $\{c_i,d_i,d_{i-1}\}$ is a coinversion triple in $\ws$ if and only if $\{d_i,c_i,c_{i-1}\}$ is a coinversion triple in $\ws'$. 
        
Consider Figure~\ref{fig:OpOnMiddle}. The boxes on the far left show the setup in $\ws$, the boxes in the middle show an intermediate step, and the boxes on the far right show the setup in $\ws'$. The orientation of the triples $\{c_i,d_i,d_{i-1}\}$ in $\ws$ is opposite the orientation of the triple of boxes $\{c_i,d_i,c_{i-1}\}$ in the middle step, and the boxes in the middle step have opposite orientation to the triple $\{d_i,c_i,c_{i-1}\}$ in $\ws'$. Since the orientation switches direction twice, the direction of increase of the triple $\{c_i,d_i,d_{i-1}\}$ in $\ws$ is the same as the direction of increase of the triple $\{d_i,c_i,c_{i-1}\}$ in $\ws'$. Therefore, the number of coinversions at the point does not change. 

        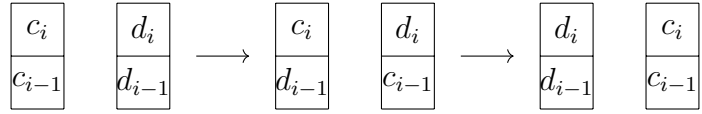
\begin{figure}[h]
            \centering
            \begin{tikzpicture}[scale=0.7]
            \begin{scope}
                \draw[] (0,0) grid (1,2);
                \draw[] (2,0) grid (3,2);

                \node[] at (0.5,0.5) {$c_{i-1}$};
                \node[] at (0.5,1.5) {$c_{i}$};
                \node[] at (2.5,0.5) {$d_{i-1}$};
                \node[] at (2.5,1.5) {$d_{i}$};

                \draw[->] (3.5,1) -- (4.5,1);
            \end{scope}
            \begin{scope}[xshift = 5cm]
                \draw[] (0,0) grid (1,2);
                \draw[] (2,0) grid (3,2);

                \node[] at (0.5,0.5) {$d_{i-1}$};
                \node[] at (0.5,1.5) {$c_{i}$};
                \node[] at (2.5,0.5) {$c_{i-1}$};
                \node[] at (2.5,1.5) {$d_{i}$};

                \draw[->] (3.5,1) -- (4.5,1); 
            \end{scope}
            \begin{scope}[xshift = 10cm]
                \draw[] (0,0) grid (1,2);
                \draw[] (2,0) grid (3,2);

                \node[] at (0.5,0.5) {$d_{i-1}$};
                \node[] at (0.5,1.5) {$d_{i}$};
                \node[] at (2.5,0.5) {$c_{i-1}$};
                \node[] at (2.5,1.5) {$c_{i}$}; 
            \end{scope}
        \end{tikzpicture}
            \caption{Action of operator on the entries $c_i$, $c_{i-1}$, $d_i$, and $d_{i-1}$}
            \label{fig:OpOnMiddle}
        \end{figure}

\item Suppose $c_i = c_{i-1}$ and $d_i = d_{i-1}$. Since the entries are equal, it is clear that the number of coinversions does not change. 

\item Assume that $c_i = c_{i-1}$, $d_i \neq d_{i-1}$, and the triple $\{c_i,d_i,d_{i-1}\}$ is a coinversion triple in $\ws$. Then, the swap results in a loss of the coinversion triple $\{c_i,d_i,d_{i-1}\}$.

\item Assume that $d_i = d_{i-1}$, $c_i \neq c_{i-1}$, and the triple $\{d_i,c_i,c_{i-1}\}$ in $\ws'$ is a coinversion triple. Then, the swap results in gaining the coinversion triple $\{d_i,c_i,c_{i-1}\}$.

\item The only other possibilities are that $c_i = c_{i-1}$, $d_i \neq d_{i-1}$, and the triple $\{c_i,d_i,d_{i-1}\}$ is a not coinversion triple in $\ws$ or that $d_i = d_{i-1}$, $c_i \neq c_{i-1}$, and the triple $\{d_i,c_i,c_{i-1}\}$ in $\ws'$ is a not coinversion triple. In either case, the swap does not affect the number of coinversion triples.
\end{enumerate}
\end{proof}

Now we consider entries $x$ and $v$ in row $j$. We show that the number of descents can change by at most one. Additionally, if the number of descents increases by one (or decreases by one), then the triple $(x,v,w)$ is counterclockwise in $\ws$ (or the triple $(x,v,w)$ is clockwise in $\ws$).

\begin{lemma} \label{lemma: ChangeByTwoDescentsNotAllowed}
    Consider the entries $x,u,v,w$ in $\ws$ and $\ws'$ as in Figure~\ref{fig:OperatorAction}. The number of descents can change by at most one in these positions. Furthermore, if a descent is gained, then $\{x,v,w\}$ is a clockwise triple in $\ws$, and, if the number of descents decreases by one, then $\{x,v,w\}$ is a counterclockwise triple in $\ws$.
\end{lemma}

\begin{proof}
    Suppose, for the sake of contradiction, that the number of descents changes by two. The number of descents can only increase by two if $x < u$, $u < v$, $v < w$, and $w < x$. The first pair of inequalities combine to show $x < v$, and the latter pair of inequalities shows that $v < x$, a contradiction. Thus, the number of descents could not increase by two.

    Similarly, suppose that the number of descents decreases by two. Then, we must have $u < x$, $w < v$, $x < w$, and $v < u$. As a result, $u < x$ and $x < w$ implies $u < w$ while $w < v$ and $v < u$ implies $w < u$, a contradiction. Thus, the number of descents cannot decrease by two.

    Therefore, the number of descents can change by $-1$, 0, or 1.

    Suppose that the number of descents increased by 1. Then one of the following scenarios is true:
    \begin{enumerate}
        \item $x \leq u$, $v < w$,  $x > w$, and $v < u \implies v < w < x$.
        \item  $x \leq u$, $v < w$,  $x < w$, and $v > u \implies x < v < w $.
        \item $x \leq u$, $v > w$, $x > w$, and $v > u \implies w < x < v$.
        \item $x > u$, $v < w$, $x > w$, and $v > u \implies v < w < x $.
    \end{enumerate}
    In each case, the triple $\{x,v,w\}$ is a clockwise triple in $\ws$.

    Suppose that the number of descents decreased by 1. Then one of the following scenarios is true:
    \begin{enumerate}
        \item $x \leq u$, $v > w$,  $x < w$, and $v < u \implies x< w < v$.
        \item  $x > u$, $v < w$,  $x < w$, and $v < u \implies v < x < w $.
        \item $x > u$, $v > w$, $x < w$, and $v > u \implies x < w < v$.
        \item $x > u$, $v > w$, $x > w$ and $v < u \implies w < v < x$.
    \end{enumerate}
    In each case, the triple $\{x,v,w\}$ is a counterclockwise triple in $\ws$.
    
\end{proof}

If $x = u$, then the change in descents is still further restricted.  
\begin{lemma} \label{lemma: ChangeInDescentsTrivialPair1}
    Consider the entries $x,u,v,w$ in Figure~\ref{fig:OperatorAction}. If $x = u$, then the number of increases by one if and only if $\{x,v,w\}$ is a clockwise triple in $\ws$, and the number of descents stays the same if and only if $\{x,v,w\}$ is a counterclockwise triple in $\ws$. 
\end{lemma}
\begin{proof}
    Suppose that $x = u$ in Figure~\ref{fig:OperatorAction}. 
    In Table~\ref{tab:NumOfDesc} we see that the number of descents increases by 1 if and only if the triple is increasing in the clockwise direction in $\ws$, and that the number of descents stays the same if and only if the triple is increasing in the counterclockwise direction in $\ws$.
    \begin{table}[h!]
        \centering
        \begin{tabular}{c|c|c}
        Order  & $\ws$ & $\ws'$ \\
        \hline
        $u < v < w$  & 1 & 2 \\
        $w < u < v$  & 1 & 2 \\
        $v < w < u$  & 0 & 1 \\
        \hline 
        $v < u < w$  & 0 & 0 \\
        $w < v < u$  & 1 & 1 \\
        $u < w < v$  & 1 & 1
        \end{tabular}
        \caption{Number of descents at positions $(i,j)$ and $(i',j)$ in the fillings $\ws$ and $\ws'$ respectively.}
        \label{tab:NumOfDesc}
    \end{table}
\end{proof}

 We can now state the total change in coinversion and major index due to $\T{i}{j}$.

\begin{prop} \label{prop: NonTrivialPairSwap}
    Suppose each entry $x, v, u, w$ are as required to apply $\mathcal{T}_{(i,j)}$ to $\ws$ as in Figure~\ref{fig:OperatorAction}. Let $\pi_{\ws}$ be the number of trivial pairs $c_i$ and $c_{i-1}$ such that $c_i$,$d_i$, and $d_{i-1}$ is a coinversion triple in $\ws$ and $\pi_{\ws'}$ be the number of trivial pairs $d_i$ and $d_{i-1}$ such that $d_i$, $c_i$, and $c_{i-1}$ forms a coinversion triple in $\ws'$:
    \begin{enumerate}
        \item If the number of descents increases by one, then $\Delta \coinv = a(i,j) + 1 + \pi_{\ws'} - \pi_{\ws} = a(i,j') + \pi_{\ws'} - \pi_{\ws}$ and $\Delta \maj = l(i,j)+1$.
        \item If the number of descents stays the same and $x,v,w$ is a coinversion triple in $\ws$, then $\Delta \coinv = -1 + \pi_{\ws'} - \pi_{\ws}$ and $\Delta \maj = 0$.
        \item If the number of descents decreases by one, then $\Delta \coinv = -a(i,j) - 1  + \pi_{\ws'} - \pi_{\ws} = -a(i,j')+ \pi_{\ws'}  - \pi_{\ws} $ and $\Delta \maj = -l(i,j)-1$.
        \item[(4)] If the number of descents stays the same and $x,v,w$ is not a coinversion triple in $\Delta \coinv = 1  + \pi_{\ws'} - \pi_{\ws}$ and $\Delta \maj = 0$.
    \end{enumerate}
\end{prop}

\begin{proof}

First, note that the operators $\T{i}{j}$ are involutions. As a result, $(1)$ holds if and only if $(3)$ holds because $\T{i}{j}(\ws') = \ws$ gains (loses) a descent if and only if $\T{i}{j}(\ws)$ loses (gains) a descent. That is, if $\ws'$ has one less descent than $\ws$, then $\Delta \coinv = \coinv(\ws') - \coinv(\ws) = -(\coinv(\ws) - \coinv(\ws')$. Then, assuming $(1)$, we have $\T{i}{j}(\ws') = \ws$ has one more descent than $\ws'$; hence, $\coinv(\ws) - \coinv(\ws') = a(i,j) + 1 + \pi_{\ws} - \pi_{\ws'}$ implying $\coinv(\ws') - \coinv(\ws) = -a(i,j) - 1 + \pi_{\ws'} - \pi_{\ws}$. Similar reasoning works for the change in major index. Thus, $(1)$ implies $(3)$. In the same way, $(2)$ implies $(1)$. Thus, it is enough to show that $(1)$ and $(2)$ hold. 

For the proof, we first show that the major index changes according to $(1)$ and $(2)$. Then, we show that the coinversion changes as claimed by breaking it up into two cases. 

The major index is defined as a sum over the descents of a filling. Since both columns have the same height, only the change in the number of descents matters. By Lemmas~\ref{lem:OperatorTop} and~\ref{lem:OperatorMiddle}, the major index is unchanged in each row above row $i$, which is only possible if the number of descents does not change in these rows. Thus, only the boxes in row $i$ that contain entries $x$ and $v$ can change the number of descents. Since both columns have the same height, the change in major index is either $0$ or $l(i,j)+1$, depending on whether the number of descents stays the same or increases by one.  This shows that the $\Delta \maj$ behaves as claimed in parts one and two. 

Notice that the change in coinversions is the sum of the change in coinversions at each row. By Lemma~\ref{lem:OperatorMiddle}, each of the middle rows contributes a loss of a coinversion to $\Delta \coinv$ if $c_i = c_{i-1}$ and $\{c_i, d_{i}, d_{i-1}\}$ is a type I coinversion triple in $\ws$. Similarly, by the same lemma, each middle row adds one to $\Delta \coinv$ if $d_i$ and $d_{i-1}$ form a trivial pair and $d_i$, $c_{i}$, and $c_{i-1}$ form a coinversion in $\ws'$. The total contribution from the middle entries is $\pi_{\ws'} -\pi_{\ws}$. The remaining change in coinversion is due to swapping $x$ and $u$. The remaining proof is split into two parts to determine the change in coinversions due to transposing $x$ and $u$ at row $i$. Each part focuses on where the descents are located in $\ws$ and $\ws'$ and the inversions related to the entries $x,v,u,w$. 

\begin{enumerate}
    \item
    Suppose the number of descents increases by one. Then the triple $x,v,w$ is not a coinversion by lemma~\ref{lemma: ChangeByTwoDescentsNotAllowed}. Since only one descent can be gained at a time by Lemma~\ref{lemma: ChangeByTwoDescentsNotAllowed}, we break up this case into sub-cases depending on where the descent is gained.
        
    Assume that a descent is gained at $(i,j)$. Then $x \leq u$. Note that $x = u$ is possible because the number of descents increases by one. Since a descent is gained, there is at least an increase by $a(i,j)$ in the change in coinversions. It remains to count the change in inversions. There are two ways to gain a descent at $(i,j)$: first, $u < v$ and $v<w$; second, $u< v$, $w < v$, and $w < x$.

    \begin{enumerate}
        \item Let $u < v$ and $v< w$, then $x < v$. As a result, $x$ and $v$ form an inversion in $\ws$ but $x$ and $w$ do not. Also, $v$ and $w$ do not form an inversion in $\ws'$. So, $x$ and $v$ do not form an inversion in $\ws'$. A decrease in the inversion set by one contributes a plus one to the change in coinversion. Thus, the contribution to $\Delta \coinv$ is $a(i,j)+1  = a(i,j')$. 
        \item Let $u < v$, $w < v$, and $w < x$, then $x < v$. Like previously, the only change in inversions is between $x$ and $v$, which means $\ws'$ has one less inversion. Thus, the contribution to $\Delta \coinv$ is $a(i,j)+1  = a(i,j')$.
    \end{enumerate}

    Assume that a descent is gained at $(i,j')$. Then $v < w$. Since the descent is gained at $(i,j')$, there is a contribution of $a(i,j') = a(i,j) + 1$ to the change in coinversions. There are two ways to gain descent at $(i,j')$ either $x \leq u$ and $w < x$ or $u<x$, $w<x$, and $u < v$.

    \begin{enumerate}
        \item Let $x \leq u$ and $w< x$ then $v < x$. There is one inversion involving $x$ and $w$ in $\ws$, and there is one inversion involving $v$ and $x$ in $\ws'$. Thus, there is no change in the number of inversions. Thus, the contribution to $\Delta \coinv$ is $a(i,j)+1  = a(i,j')$. 
        \item Let $u<x$, $w<x$, and $u < v$ then $v < x$. The pair $x$ and $w$ form an inversion in $\ws$, and the pair $x$ and $u$ form an inversion in $\ws'$. This is every inversion involving $x$,$u$,$v$,$w$. Thus the change due to the change of inversions is 0. Thus, the contribution to $\Delta \coinv$ is $a(i,j)+1  = a(i,j')$.
    \end{enumerate}

    This completes the proof for the first part. Therefore, $\Delta \coinv = a(i,j) + 1 + \pi_{\ws'} - \pi_{\ws} = a(i,j') + \pi_{\ws'} - \pi_{\ws}$ and $\Delta \maj = l(i,j)+1$.

    \item Suppose the number of descents stays the same and entries $x$, $v$, and $w$ form a coinversion triple in $\ws$. The change in coinversions depends upon the manner in which the descents stay the same. The number of descents in $\ws$ can be 0, 1, or 2 in positions $(i,j)$ and $(i,j')$. We investigate each possibility below.

    \begin{enumerate}
        \item Assume that the number of descents in $\ws$ is 0 then the number of descents in $\ws'$ is also 0. Thus, $x \leq u$, $v < w$, $v < u$, and $x < w$. Since $x$, $v$, and $w$ form a counterclockwise triple in $\ws$ and both $x$ and $v$ are less than $w$, $v < x < w$. Since there are no descents in $\ws$ and $\ws'$ at $(i,j)$ and $(i,j')$, there is no contribution to the change in coinversions coming from the arm of either box. It remains to count the change in inversions. In $\ws$, there are no inversions. In $\ws'$, the entries $x$ and $v$ form the only inversion; hence, the number of inversions increases by one. As a result, the contribution to the change in coinversions is $-1$.
        
        \item Assume that the number of descents in $\ws$ is 1. The descent is either at $(i,j)$ or $(i,j')$. Since $x,v,w$ forms a coinversion triple in $\ws$, either $v < x < w$ or $x < w < v$. For $\ws'$, the descent is in position $(i,j)$ in the former case, and in the latter case the descent is in position $(i,j')$.
            
        Suppose $v < x < w$. Then only the change in inversions contributes to the change in coinversions. In $\ws$, there are no inversions. In $\ws'$, $x$ and $v$ form an inversion. Since the inversions increased by 1, the coinversions decreased by 1. 

        Suppose $x < w < v$ then descent is at $(i,j')$ in $\ws$ and at $(i,j)$ in $\ws'$. In $\ws$, the descent contributes $a(i,j') = a(i,j)+1$ to the coinversions. In $\ws'$, the descent contributes $a(i,j)$ to the coinversions. Since the descent is in a different position, the contribution to the change in coinversions so far is $a(i,j) - (a(i,j)+1) = -1$. As for the inversions, notice that $x$ and $v$ form an inversion in $\ws$ and $v$ and $w$ form an inversion in $\ws'$. Thus, the number of inversions does not change. As a result, the total contribution to the change in the coinversions is $-1$.
        
        \item Assume that the number of descents in $\ws$ is 2. Since $x$, $v$ and $w$ forms a coinversion in $\ws$, the relative order must be $w < v < x$. Since there is a descent at both positions, the arms of the boxes do not contribute to the change in the coinversions. Notice that $x$ and $w$ form an inversion in $\ws$ while $x$ and $v$ as well as $v$ and $w$ form an inversion in $\ws'$. Since the number of inversions increased by one, the number of coinversion decreased by 1.
        \end{enumerate}
        This completes all the cases for part two. Thus, $\Delta \coinv = -1 + \pi_{\ws'} - \pi_{\ws}$ and $\Delta \maj = 0$     
    \end{enumerate}
    
\end{proof}

In the particular, if $x=u$ in Figure~\ref{fig:OperatorAction}, then only $(1)$ or $(2)$ occurs in Proposition~\ref{prop: NonTrivialPairSwap}.
\begin{corollary} \label{Cor:trivialpair}
    Suppose $\ws$ is a filling so that $\mathcal{T}_{(i,j)}(\ws) = \ws'$ where $\ws(i,j) = \ws(i-1,j) = u$. 
    Let $\pi_{\ws}$ be the number of trivial pairs $c_i$ and $c_{i-1}$ such that $c_i$,$d_i$, and $d_{i-1}$ 
    is a coinversion triple in $\ws$ and $\pi_{\ws'}$ be the number of trivial pairs $d_i$ and $d_{i-1}$ 
    such that $d_i$, $c_i$, and $c_{i-1}$ forms a coinversion triple in $\ws'$. Then, we have the following cases:
    \begin{enumerate}
        \item If the number of descents increases by one, then $\Delta \coinv = a(i,j) + 1 + \pi_{\ws'} 
        - \pi_{\ws} = a(i,j') + \pi_{\ws'} - \pi_{\ws}$ and $\Delta \maj = l(i,j)$.
        \item If the number of descents stays the same and $x,v,w$ is a coinversion triple in $\ws$, then 
        $\Delta \coinv = -1 + \pi_{\ws'} - \pi_{\ws}$ and $\Delta \maj = 0$.
    \end{enumerate}
    
\end{corollary}

\begin{proof}
    By Lemma~\ref{lemma: ChangeInDescentsTrivialPair1}, the descents can only increase by one if and only if $\{x,v,w\}$ increases clockwise, or the number of descents stays the same if and only if $\{x,v,w\}$ increases counterclockwise. This limits the possible change in coinversions and major index to the first two cases of Proposition~\ref{prop: NonTrivialPairSwap}.
\end{proof}

In Section 6, we apply the operators $\T{i}{j}$ to certain mixed nonattacking fillings. Acting on these mixed nonattacking fillings gets rid of the extra $\pi_{\ws'}-\pi_{\ws}$ term occurring in the change in coinversions. As a result, the proof becomes easier to track. 
\begin{definition}\label{def:SpecialRestFunc}
    A \emph{special restriction function} $R$ is a restriction function such that 
    \begin{enumerate} \label{list: ConditionsOnR}
    \item $R(i',j') = R(i,j') \cup \{(i,j'-1)\}$ for all $\mu_{i} \ge j' > j$ and 
    \item $R(i,j) = R(i',j)$.
\end{enumerate}
\end{definition}
Let $\ws$ be a mixed nonattacking filling for a special restriction function $R$. 
\begin{remark}
    The action of the operator $\T{i}{j}$ results in a mixed-nonattacking filling since it preserves trivial pairs in every box except $(i,j)$.
\end{remark}
By the choice of $R$, trivial pairs in columns $i$ and $i'$ above row $j$ of $\ws$ or $\ws'$ do not count towards coinversion triples. As a result, we have the following corollary.

\begin{corollary} \label{cor:ChangeInCandMSpFunc}
    Suppose each entry $x, v, u, w$ is as required to apply $\mathcal{T}_{(i,j)}$ to the mixed nonattacking filling $\ws$ as in Figure~\ref{fig:OperatorAction}. Then, the change to the coinversion triples and major index is one of the following:
    \begin{enumerate}
        \item If the number of descents increases by one, then $\Delta \coinv = a(i,j) + 1 = a(i,j')$ and $\Delta \maj = l(i,j)+1$.
        \item If the number of descents stays the same and $x,v,w$ is a coinversion triple in $\ws$, then $\Delta \coinv = -1$ and $\Delta \maj = 0$.
        \item If the number of descents decreases by one, then $\Delta \coinv = -a(i,j) - 1  = -a(i,j') $ and $\Delta \maj = -l(i,j)-1$.
        \item If the number of descents stays the same and $x,v,w$ is not a coinversion triple in $\Delta \coinv = 1 $ and $\Delta \maj = 0$.
    \end{enumerate}
    In particular, if $x = u$, then only $(1)$ or $(2)$ are possible.
\end{corollary}

\section{Compressing Nonattacking Fillings} \label{section:Compressing}

In this section, we prove Theorem~\ref{thm:MainResult}. The compositions $\mu$ that we consider satisfy the condition that each box in $\dg'(\mu)$ hits at most one box. That is, $\mu$ is a composition such that $\mu_i = \mu_{i'}$ where $i < i'$ and $\mu_c < j-1$ for all $i < c < i'$ where $1 \le j \le \mu_i$. If $i' = i+1$, then the $1 \le j \le \mu_i$. Let $j$ be such that $(i,j)$ is a box in column $i$ with the property that $\mu_c < j-1$ for all $i < c < i'$.   Recall that, given $\ws$, we write $\T{i}{j})(\ws) = \ws'$ whenever the context is clear. Throughout this section we refer the reader to Figure~\ref{fig:OperatorAction} for a generic setup of this.

Before we prove Theorem~\ref{thm:MainResult}, we show how the sum of the $q,t$-weights of $\ws$ and $\ws'$ simplifies. In the computation that follows, we assume that $x= u$, so that they form a trivial pair, and that $\ws$ is a mixed nonattacking filling for a special restriction function $R$ as in Definition~\ref{def:SpecialRestFunc}. We show how a pair of mixed nonattacking fillings for $R$ compress into a single term.
=

If $\{u,v,w\}$ is not a coinversion triple in $\ws$, then by Corollary~\ref{cor:ChangeInCandMSpFunc} we have
\begin{align*} 
    \wt_m(\ws') + \wt_m(\ws) &= \Big(1 + q^{l(i,j)+1}t^{a(i,j) + 1 -t(i,j)} \dfrac{1-t}{1-q^{l(i,j)+1}t^{a(i,j)+1-t(i,j)}} \Big)\wt_m(\ws) \\ &=\Big( \dfrac{1-q^{l(i,j)+1}t^{a(i,j)+1-t(i,j)}}{1-q^{l(i,j)+1}t^{a(i,j)+1-t(i,j)}} + \dfrac{q^{l(i,j)+1}t^{a(i,j) + 1 -t(i,j)}-q^{l(i,j)+1}t^{a(i,j) + 1 - t(i,j)+1}}{1-q^{l(i,j)+1}t^{a(i,j)+1-t(i,j)}} \Big)\wt_m(\ws)  \\&= \Big( \dfrac{1-q^{l(i,j)+1}t^{a(i,j)+1-t(i,j)+1}}{1-q^{l(i,j)+1}t^{a(i,j)+1-t(i,j)}} \Big) \wt_m(\ws) \\
    &= \Big( \dfrac{1-q^{l(i,j)+1}t^{a(i',j)+1-t(i',j)}}{1-q^{l(i,j)+1}t^{a(i,j)+1-t(i,j)}} \Big) \wt_m(\ws).
\end{align*}
Notice that the numerator $1-q^{l(i,j)+1}t^{a(i',j)+1-t(i',j)}$ is a common factor of the denominator of $\wtqt(\ws)$. Specifically, it is the contribution to the product by box $(i',j)$. The last equality holds because the arm of $(i',j)$ is exactly one larger than the arm of $(i,j)$, $a(i,j) + 1 = a(i',j)$, and $t(i,j) = t(i',j)$ since $R(i,j) = R(i',j)$. 

Similarly, if $\{u,v,w\}$ is a coinversion triple, then by Corollary~\ref{cor:ChangeInCandMSpFunc} we have
\begin{align*} 
    \wt_m(\ws') + \wt_m(\ws) &= \Big(1 + t^{-1} \dfrac{1-t}{1-q^{l(i,j)+1}t^{a(i,j)+1-t(i,j)}} \Big)\wt_m(\ws) \\ &=  t^{-1} \Big(t + \dfrac{1-t} {1-q^{l(i,j)+1}t^{a(i,j)+1-t(i,j)}} \Big)\wt_m(\ws) \\ &= \Big( \frac{t -q^{l(i,j)+1}t^{a(i,j)+1-t(i,j) + 1}}{1-q^{l(i,j)+1}t^{a(i,j)+1-t(i,j)}} + \frac{1-t}{1-q^{l(i,j)+1}t^{a(i,j)+1-t(i,j)}} \Big)\wt_m(\ws) \\ &= t^{-1} \Big( \dfrac{1-q^{l(i,j)+1}t^{a(i,j)+1-t(i,j)+1}}{
    1-q^{l(i,j)+1}t^{a(i,j)+1-t(i,j)}} \Big) \wt_m(\ws) \\ &= t^{-1} \Big( \dfrac{1-q^{l(i,j)+1}t^{a(i',j)+1-t(i',j)}}{
    1-q^{l(i,j)+1}t^{a(i,j)+1-t(i,j)}} \Big) \wt_m(\ws).
\end{align*}
Like the previous case, the numerator term cancels with the corresponding denominator term and, unlike the previous case, the power of $t$ decreases by 1. The $t^{-1}$ removes the coinversion triple $\{x,v,w\}$ from the set of coinversions as depicted in Figure~\ref{fig:OperatorAction}.  

The main idea to prove Theorem~\ref{thm:MainResult} is to use the previous computation to show that two explicit sets of mixed nonattacking fillings generate the same polynomial. Suppose $R$ is a special restriction function on $\dg'(\mu)$ as specified by Definition~\ref{def:SpecialRestFunc}.
We construct a different special restriction function $R'$ that depends on $R$:
\begin{enumerate}
    \item $R'(u) = R(u)$ for $u \neq (i',j) \in \dg'(\mu)$ and 
    \item $R(i',j) = R(i,j) \cup \{(i,j-1)\}$.
\end{enumerate}
The resulting mixed nonattacking fillings are $\MF(\mu,R(u))$ and $\MF(\mu,R'(u))$. Then, the following proposition holds.

\begin{prop} \label{prop:InductionStep}
    The generating function for mixed nonattacking fillings $\MF(\mu,R(u))$ equals the generating function for the mixed nonattacking fillings for $\MF(\mu,R'(u))$. That is,
    $$ \sum_{\ws \in \MF(\mu,R(u))} \wt_m(\ws)x^{\sigma} = \sum_{\ws \in \MF(\mu,R'(u))} \wt_{m'}(\ws)x^{\sigma},$$
    where $\wt_{m'}(\ws)$ is the mixed $q,t$-weight for the fillings in $\MF(\mu,R'(u))$.
\end{prop}
\begin{proof}
    We simplify the notation by letting $\MF(\mu,R(u)) = \MF$ and $\MF(\mu,R'(u)) = \MF'$. We partition the set $\MF$ into a union of three disjoint sets $\MF_1$, $\MF_2$, and $\MF_3$ as follows:
    \begin{enumerate}
        \item $\MF_1 := \{ \ws \in \MF\ \mid\ \ws(i,j) = \ws(i,j-1), \ws(i',j) \neq (\ws(i',-1)\}$,
        \item $\MF_2 := \{ \T{i}{j}(\ws) = \ws' \in \MF \ \mid\ \ws \in \MF_1\}$, and
        \item $\MF_3 = \{ \ws \in \MF\ \mid\ \ws(i',j) = \ws(i',j-1)\} \cup \{ \ws \in \MF\ \mid\ \ws(i',j) \neq \ws(i',j-1) \neq \ws(i,j) \neq \ws(i,j-1) \}$.
    \end{enumerate}
    The sets $\MF_1$ and $\MF_3$ partition $\MF'$ into a union of two disjoint sets. Our goal is to show that 
    \begin{align*}
        \wt_m(\ws) + \wt_m(\ws') &= \wt_{m'}(\ws) \qquad \text{ for } \ws \in \MF_1, \quad \text{ and } \\
        \wt_m(\ws) &= \wt_{m'}(\ws)  \qquad \text{ for } \ws \in \MF_3.
    \end{align*} 
    
    Notice that the associated mixed weight for $\ws \in \MF_3$ is the same for both $R$ and $R'$ because the additional box $(i,j) \in R(i',j)$ only affects the mixed weight associated to $\ws$, when viewed as a mixed nonattacking filling for $R'$, if $\ws(i,j) = \ws(i,j)$ and $\ws(i',j) \neq \ws(i',j-1)$. Therefore, $\wt_m(\ws) = \wt_m(\ws)$ for $\ws \in \MF_3$. 

    Let $\ws \in \MF_1$, then $\ws' \in \MF_2$. Define $\delta = 0$ if $\{x,v,w\}$ is not a type II coinversion triple, and $\delta = 1$ if $\{x,v,w\}$ is a coinversion triple. Then, by the calculation above,
    \begin{equation} \label{eq:SumReduce}
       \wt_m(\ws) + \wt_m(\ws') = \Big( t^{-\delta} \dfrac{1-q^{l(i,j)+1}t^{a(i,j)+1-t(i,j)+1}}{1-q^{l(i,j)+1}t^{a(i,j)+1-t(i,j)}} \Big) \wt_m(\ws).
    \end{equation}
    \noindent Recall that 
    $$\wt_m(\ws) = \prod_{\substack{u \in \dg'(\mu) \\ \ws(u) \neq\ws(d(i',j))}} \wt_{m}(u) = \wt_m(i',j) \prod_{\substack{u \in \dg'(\mu) \\ \ws(i',j) \neq \ws(d(i',j)) \\ u \neq (i',j)}} \wt_{m}(u)$$
    and 
    $$\wt_m(i',j) = 
    \begin{cases}
        q^{l(i',j)+1}t^{c(i',j)-r(i',j)} \dfrac{1-t}{1-q^{l(i',j)+1}t^{a(i',j)+1-t(i',j)}} & \qquad \text{ if } \ws(i',j) > \ws(d(i',j)) \\
        t^{c(i',j)-r(i',j)} \dfrac{1-t}{1-q^{l(i',j)+1}t^{a(i',j)+1-t(i',j)}} & \qquad \text{ if } \ws(i',j) < \ws(d(i',j)). 
    \end{cases}
    $$
    \noindent Note that $t(i,j) = t(i',j)$ because $R(i,j) = R(i',j)$ and that $a(i,j) + 1 = a(i',j)$. Thus, 
    $$1-q^{l(i,j)+1}t^{a(i,j)+1-t(i,j)+1} = 1-q^{l(i',j)+1}t^{a(i',j)+1-t(i',j)}.$$ 
    Let $t'$ denote the trivial pair function that corresponds to the restriction function $R'$. Since $R'(i',j) = R(i',j) \cup \{(i',j)\}$, $t'(i',j) = t(i',j)+1$. As a result, the sum in~\eqref{eq:SumReduce} simplifies to 
    \begin{align*}
        q^{l(i',j)+1}t^{c(i',j)-r(i',j)- \delta} \dfrac{1-t}{1-q^{l(i',j)+1}t^{a(i',j)-t(i',j)}} \prod_{\substack{u \in \dg'(\mu) \\ \ws(i',j) \neq \ws(d(i',j)) \\ u \neq (i',j)}} \wt_{m}(u) \qquad \text{ if $\ws(i',j) > \ws(d(i',j))$ or } \\
        t^{c(i',j)-r(i',j) - \delta} \dfrac{1-t}{1-q^{l(i',j)+1}t^{a(i',j)-t(i',j)}} \prod_{\substack{u \in \dg'(\mu) \\ \ws(i',j) \neq \ws(d(i',j)) \\ u \neq (i',j)}} \wt_{m}(u) \qquad \text{ if $\ws(i',j) < \ws(d(i',j))$. }
    \end{align*}
    Now, observe that 
    $$\wt_{m'}(i',j) = 
    \begin{cases}
        q^{l(i',j)+1}t^{c(i',j)-r(i',j) - \delta} \dfrac{1-t}{1-q^{l(i',j)+1}t^{a(i',j)-t(i',j)}} & \qquad \text{ if } \ws(i',j) > \ws(d(i',j)), \\
        t^{c(i',j)-r(i',j) - \delta} \dfrac{1-t}{1-q^{l(i',j)+1}t^{a(i',j)-t(i',j)}} & \qquad \text{ if } \ws(i',j) < \ws(d(i',j)). 
    \end{cases}
    $$
    since 
    $$a(i',j) - t(i',j) = a(i',j) + 1 - (t(i',j) + 1) = a(i',j) + 1 - t'(i',j).$$ 
    Also, notice that $\wt_m(u) = \wt_{m'}(u)$ for $u \in \dg'(\mu)$ and $u \neq = (i',j)$ since $R(u) = R'(u)$. 
    
    Finally, we see that 
    \begin{align*}
        \wt_m(\ws) + \wt_m(\ws') &= \Big( t^{-\delta} \dfrac{1-q^{l(i,j)+1}t^{a(i,j)+1-t(i,j)+1}}{1-q^{l(i,j)+1}t^{a(i,j)+1-t(i,j)}} \Big) \wt_m(\ws)\\
        &=  \Big( t^{-\delta} \dfrac{1-q^{l(i,j)+1}t^{a(i,j)+1-t(i,j)+1}}{1-q^{l(i,j)+1}t^{a(i,j)+1-t(i,j)}} \Big) \wt_m(i',j) \prod_{\substack{u \in \dg'(\mu) \\  \ws(i',j) \neq \ws(d(i',j)) \\ u \neq (i',j)}} \wt_{m}(u) \\
        &= \wt_{m'}(i',j) \prod_{\substack{u \in \dg'(\mu) \\  \ws(i',j) \neq \ws(d(i',j)) \\ u \neq (i',j)}} \wt_{m'}(u) \\
        &= \wt_{m'}(\ws).
    \end{align*}
    
    Therefore, 
    \begin{align*}
        \sum_{\ws \in \MF} \wt_m(\ws) x^{\sigma} &= \sum_{\ws \in \MF_1} \wt_m(\ws) x^{\sigma} + \sum_{\ws' \in \MF_2} \wt_m(\ws') x^{\sigma} + \sum_{\ws \in \MF} \wt_m(\ws) x^{\sigma} \\
        &= \sum_{\ws \in \MF_1} \wt_{m'}(\ws) x^{\sigma} + \sum_{\ws \in \MF_3} \wt_{m'}(\ws) x^{\sigma} 
        \\ &= \sum_{\ws \in \MF'} \wt_{m'}(\ws) x^{\sigma}.
    \end{align*}
    
\end{proof}

Now that we have all the required results to prove our main result, we recall the statement here for convenience.
\renewcommand{\thetheorem}{\ref{thm:MainResult}}
\begin{theorem}
Suppose $\mu = (\mu_1, \ldots, \mu_n)$ is a composition such that each box $u \in \dg'(\mu)$ hits at most one box. Then,
$$E_{\mu} = \sum_{\ws \in \mathcal{R}(\mu)} \wt_{m}(\ws)x^{\sigma}.
$$
\end{theorem} 

\begin{proof}
    We partition the diagram $\dg'(\mu)$ into disjoint rectangular regions given by the southwest and northeast corners. That is, each region is a set $\{(c,r) \mid i' \le c \le i, j \le r \le j'\}$, where $(i,j), (i',j') \in \dg'(\mu)$. In our case, we partition $\dg'(\mu)$ into pairs of columns $i$ and $i'$ such that 
    $\lambda_i = \lambda_{i'}$ and there exists a smallest $1 \le j \le \mu_i$ such that $\mu_c \le j-2$ for $i < c <i'$. Such a block is given by the corners $(i,j)$ and $(i',\mu_{i'})$. See Figure~\ref{fig:OperatorAction} for an example of this type of block. If there does not exist a smallest $j$, then the block is given by the corners $(i,1)$ and $(i,\mu_i)$. The proof is by induction.

    To begin, notice that if a block consists of a single column, then there is nothing to prove. It is enough to just consider blocks that consist of two columns. Suppose $B_{i_1}, B_{i_2}, \ldots, B_{i_l}$ are the blocks consisting of two columns, with the left column of block $B_{i_c}$ being part $\mu_{i_c}$; as a result, $i_1 < i_2 < \ldots < i_l$. Denote the right column block of $B_{i_c}$ by ${i'_c}$ for $1 \le c \le l$. We define the following linear order on the boxes in columns $i_c$ for $1 \le c \le l$. For $(i_c,j) \in B_{i_c}$ and $(i_{c'},j) \in B_{i_{c'}}$, we say $(i_c, j) < (i_{c'}, j')$ if $i_c < i_{c'}$ or, in the case that $i_c = i_{c'}$, if $j > j'$. That is, the smallest box is the top-leftmost box in $B_{i_1}$, and the largest box is the bottom-rightmost box in $B_{i_l}$.

    Now, we define a sequence of restriction functions on the diagram of $\mu$ recursively. First, let $R$ be the restriction function such that $R(u) = \emptyset$ for all $u \in \dg'(\mu)$; the resulting mixed nonattacking fillings and mixed $q,t$-weight is the same as the $\HHL$ model. 
    Define $R^{(i_1,\mu_{i_1})}(u) = R(u)$ for all $u \neq (i_1,\mu_{i_1}) \in \dg'(\mu)$ and $R^{(i_1,\mu_{i_1})}(i',\mu_{i'}) = \{(i_1,\mu_{i_1}-1)\}$. 
    Suppose $(i_c,j)$ covers the box $(i_{c'},j')$ based on the ordering given above. Then, given $R^{(i_c,j)}$, we define $R^{(i_{c'},j')}$ to be the restriction function on $\dg'(\mu)$ such that $R^{(i_{c'},j')}(u) = R^{(i_{c},j)}(u)$ for $u \neq (i'_{c'},j')$ and $R^{(i_{c'},j')}(i'_{c'},j') = (i_{c'},j'-1)$. The idea is that, at each step, we restrict one additional box. Denote the set of mixed fillings $\MF(\mu,R^{(i_c,j)}(u))$ by $\MF^{(i_c,j)}$, and denote the corresponding mixed $q,t$-weight by $\wt_m^{(i_c,j)}$. 

    If $(i_l,j)$ is the final box in the ordering, our goal is to prove 
    \begin{equation} \label{eq:Claim}
    E_\mu = \sum_{\ws \in \MF^{(i_l,j)}} \wt_m^{(i_l,j)}(\ws) x^{\sigma},
    \end{equation}
    since $\MF^{(i_l,j)} = \mathcal{R}(\mu)$ and the mixed $q,t$-weights are equivalent. We proceed by induction on the boxes of the form $(i_c,j)$ in the order defined above. 
    
    The base case is the following claim: $$E_\mu = \sum_{\ws \in \MF(\mu, R^{(i_1,\mu_1)}(u))} \wt_m^{(i_1,\mu_1)}(\ws) x^{\sigma}.$$ Notice that this is the same setup as in Proposition~\ref{prop:InductionStep} but $R'$ is $ R^{(i_1,\mu_1)}(u))$. Thus, the base case is true.

    For the induction hypothesis, suppose the weight generating functions corresponding to each $R^{(i_c,j)}$ are all equivalent except for the last box $(i_l,j)$ in the order. Without loss of generality, suppose that the box before $(i_l,j)$ in the order is $(i_l,j+1)$.  We show that $$ \sum_{\ws \in \MF(\mu,R^{(i_l,j+1)}(u))} \wt_m^{(i_l,j+1)}(\ws)x^{\sigma} = \sum_{\ws \in \MF(\mu,R^{(i_l,j)}(u))} \wt_{m}^{(i_l,j)}(\ws)x^{\sigma}.$$ The previous equality follows by Proposition~\ref{prop:InductionStep}, proving the induction step. Therefore, Equation~\eqref{eq:Claim} is true, which proves the statement.
 
\end{proof}

\section{Conclusion}\label{sect:Conclusion}

In this section we compare the restricted fillings model with the $\HHL$ and $\MLQ$ models. In particular, we note that the number of restricted fillings for a partition $\lambda$ equals the number of multiline queues of type $\lambda$. Also, we explain the difficulty of extending our proof methods to more than two columns. We do have some ideas to get around this difficulty and are currently working on them.

\subsection{Number of Nonattacking Fillings and Restricted Nonattacking Fillings}

For compositions in general, the number of restricted nonattacking fillings can be significantly smaller compared to the classical nonattacking fillings. We show this data in Table~\ref{tab:HHLvsRestricted}.

In comparison with the MLQ model, we notice that for partitions the number of restricted nonattacking fillings is equal to the number of multiline queues. \emph{We currently have a bijection between these two sets that is not weight-preserving, see \ref{section:RNF}.}
If Conjecture~\ref{conj:MainConj} is true, then the pair of statistics for $q$ and $t$ in the restricted nonattacking filling case is equidistributed with the corresponding statistics for $q$ and $t$ in the \MLQ model. It would be interesting to find a bijection between restricted nonattacking fillings and multiline queues for partitions that is weight-preserving.

\begin{table}[h]
    \centering
    \begin{tabular}{c|c|c}
     $\mu$              &  Nonattacking fillings   & Restricted nonattacking fillings\\
     \hline 
     $(3,3,1,1,0)$      & 64        & 27 \\
     $(3,1,1,3,0)$      & 128       & 104 \\
     $(1,1,0,3,3)$      & 1024      & 507 \\    
     $(4,4,4,2,2,0)$    & 16384     & 1632 \\
     $(4,2,4,4,2,0)$    & 36864     & 11424 \\
     $(0,2,4,4,4,2)$    & 1769472   & 113152 \\
     $(0,2,2,4,4,4)$    & 4194304   & 166464 \\
      & &
    \end{tabular}
    \caption{Comparison of the number of nonattacking fillings and restricted nonattacking fillings for several compositions. }
    \label{tab:HHLvsRestricted}
\end{table}

\subsection{Discussion About the General Case}

To understand why extending to more than two columns is difficult, we explain how we discovered the cancellation in the two-column case in an example and show why it is difficult to extend.

Let us consider the composition $\mu = (0,0,2,2,2,2)$.
In Figure~\ref{fig:NAF(0,0,2,2,2,2)}, we show a pair of nonattacking fillings whose $q,t$-weights add nicely. We build the proof of our main result, Theorem~\ref{thm:MainResult}, from this observation by creating a family of operators that mimic how the $q,t$-weights are added.

\begin{figure}[h!]
    \centering
    \begin{tikzpicture}[scale=0.5]
     \begin{scope}
         \draw[thick] (0,0) grid (6,1);
         \draw[thick] (2,1) grid (4,3);

         \node at (0.5,0.5) {$1$};
         \node at (1.5,0.5) {$2$};
         \node at (2.5,0.5) {$3$};
         \node at (3.5,0.5) {$4$};
         \node at (4.5,0.5) {$5$};
         \node at (5.5,0.5) {$6$};

         \node at (2.5,1.5) {$3$};
         \node at (3.5,1.5) {$4$};

         \node at (2.5,2.5) {$3$};
         \node at (3.5,2.5) {$1$};

         \node at (3.5,-1.25) {$t\dfrac{1-t}{1-qt^2}$};

         \node at (7,-1) {$+$};
     \end{scope}
     \begin{scope}[xshift = 8cm]
         \draw[thick] (0,0) grid (6,1);
         \draw[thick] (2,1) grid (4,3);

         \node at (0.5,0.5) {$1$};
         \node at (1.5,0.5) {$2$};
         \node at (2.5,0.5) {$3$};
         \node at (3.5,0.5) {$4$};
         \node at (4.5,0.5) {$5$};
         \node at (5.5,0.5) {$6$};

         \node at (2.5,1.5) {$3$};
         \node at (3.5,1.5) {$4$};

         \node at (2.5,2.5) {$1$};
         \node at (3.5,2.5) {$3$};

         \node at (3.5,-1.25) {$\dfrac{(1-t)^2}{(1-qt)(1-qt^2)}$};

         \node at (8.5, -1.25) {\hspace{1cm}$ = \quad \dfrac{1-t}{1-qt}$};
     \end{scope}
    \end{tikzpicture}
    \caption{Two nonattacking fillings that show how the two-column case works.}
    \label{fig:NAF(0,0,2,2,2,2)}
\end{figure}

The next step is to extend the same reasoning to longer rows; however, it is not obvious how to pair nonattacking fillings so that the sum reduces like in the case for two columns.  In Figure~\ref{fig:NAF002222}, we show an example where grouping the nonattacking fillings simplifies as in Conjecture~\ref{conj:MainConj}. Each row corresponds to a restricted nonattacking filling, with its mixed $q,t$-weight on the right. 

\begin{figure}[h!]
    \centering
    \begin{tikzpicture}[scale=0.5]
     \begin{scope}
         \draw[thick] (0,0) grid (6,1);
         \draw[thick] (2,1) grid (5,3);

         \node at (0.5,0.5) {$1$};
         \node at (1.5,0.5) {$2$};
         \node at (2.5,0.5) {$3$};
         \node at (3.5,0.5) {$4$};
         \node at (4.5,0.5) {$5$};
         \node at (5.5,0.5) {$6$};

         \node at (2.5,1.5) {$3$};
         \node at (3.5,1.5) {$4$};
         \node at (4.5,1.5) {$2$};

         \node at (2.5,2.5) {$3$};
         \node at (3.5,2.5) {$1$};
         \node at (4.5,2.5) {$5$};

         \node at (3.5,-1) {$qt^2\frac{(1-t)^2}{(1-qt^2)(1-qt^3)}$};

         \node at (7,-1) {$+$};
     \end{scope}
     \begin{scope}[xshift = 8cm]
         \draw[thick] (0,0) grid (6,1);
         \draw[thick] (2,1) grid (5,3);

         \node at (0.5,0.5) {$1$};
         \node at (1.5,0.5) {$2$};
         \node at (2.5,0.5) {$3$};
         \node at (3.5,0.5) {$4$};
         \node at (4.5,0.5) {$5$};
         \node at (5.5,0.5) {$6$};

         \node at (2.5,1.5) {$3$};
         \node at (3.5,1.5) {$4$};
         \node at (4.5,1.5) {$2$};

         \node at (2.5,2.5) {$1$};
         \node at (3.5,2.5) {$3$};
         \node at (4.5,2.5) {$5$};
         
         \node at (3.5,-1) {$qt\frac{(1-t)^3}{(1-qt)(1-qt^2)(1-qt^3)}$};

         \node at (7.5, -1) {$+$};
     \end{scope}
     \begin{scope}[xshift = 16cm]
         \draw[thick] (0,0) grid (6,1);
         \draw[thick] (2,1) grid (5,3);

         \node at (0.5,0.5) {$1$};
         \node at (1.5,0.5) {$2$};
         \node at (2.5,0.5) {$3$};
         \node at (3.5,0.5) {$4$};
         \node at (4.5,0.5) {$5$};
         \node at (5.5,0.5) {$6$};

         \node at (2.5,1.5) {$3$};
         \node at (3.5,1.5) {$4$};
         \node at (4.5,1.5) {$2$};

         \node at (2.5,2.5) {$1$};
         \node at (3.5,2.5) {$5$};
         \node at (4.5,2.5) {$3$};

         \node at (3.5,-1) {$q^2t^3\frac{(1-t)^3}{(1-qt)(1-qt^2)(1-qt^3)}$};

         \node at (11, -1) {$ \hspace{1cm} = qt\frac{(1-t)^2}{(1-qt)(1-qt^2)}$};
     \end{scope}
     \begin{scope}[yshift = -6cm]
         \draw[thick] (0,0) grid (6,1);
         \draw[thick] (2,1) grid (6,3);

         \node at (0.5,0.5) {$1$};
         \node at (1.5,0.5) {$2$};
         \node at (2.5,0.5) {$3$};
         \node at (3.5,0.5) {$4$};
         \node at (4.5,0.5) {$5$};
         \node at (5.5,0.5) {$6$};

         \node at (2.5,1.5) {$3$};
         \node at (3.5,1.5) {$4$};
         \node at (4.5,1.5) {$2$};
         \node at (5.5,1.5) {$6$};

         \node at (2.5,2.5) {$3$};
         \node at (3.5,2.5) {$5$};
         \node at (4.5,2.5) {$1$};
         \node at (5.5,2.5) {$6$};

         \node at (3.5,-1) {$qt\frac{(1-t)^2}{(1-qt^2)(1-qt^3)}$};

         \node at (7,-1) {$+$};
     \end{scope}
     \begin{scope}[xshift = 8cm, yshift = -6cm]
         \draw[thick] (0,0) grid (6,1);
         \draw[thick] (2,1) grid (5,3);

         \node at (0.5,0.5) {$1$};
         \node at (1.5,0.5) {$2$};
         \node at (2.5,0.5) {$3$};
         \node at (3.5,0.5) {$4$};
         \node at (4.5,0.5) {$5$};
         \node at (5.5,0.5) {$6$};

         \node at (2.5,1.5) {$3$};
         \node at (3.5,1.5) {$4$};
         \node at (4.5,1.5) {$2$};

         \node at (2.5,2.5) {$1$};
         \node at (3.5,2.5) {$3$};
         \node at (4.5,2.5) {$5$};

         \node at (3.5,-1) {$q\frac{(1-t)^3}{(1-qt)(1-qt^2)(1-qt^3)}$};

         \node at (7.5, -1) {$+$};
     \end{scope}
     \begin{scope}[xshift = 16cm, yshift = -6cm]
         \draw[thick] (0,0) grid (6,1);
         \draw[thick] (2,1) grid (5,3);

         \node at (0.5,0.5) {$1$};
         \node at (1.5,0.5) {$2$};
         \node at (2.5,0.5) {$3$};
         \node at (3.5,0.5) {$4$};
         \node at (4.5,0.5) {$5$};
         \node at (5.5,0.5) {$6$};

         \node at (2.5,1.5) {$3$};
         \node at (3.5,1.5) {$4$};
         \node at (4.5,1.5) {$2$};

         \node at (2.5,2.5) {$5$};
         \node at (3.5,2.5) {$1$};
         \node at (4.5,2.5) {$3$};
         
         \node at (3.5,-1) {$q^2t^2\frac{(1-t)^3}{(1-qt)(1-qt^2)(1-qt^3)}$};

         \node at (11, -1) {$ \hspace{1cm} =  q\frac{(1-t)^2}{(1-qt)(1-qt^2)}$};
     \end{scope}
    \end{tikzpicture}
    \caption{An example of the sum of $q,t$-weight simplifying in the three column case.}
    \label{fig:NAF002222}
\end{figure}

This is not the general case because simple transpositions are not enough. To see this, we focus solely on the top row. Consider the following setup of the top of $k$ columns in Figure~\ref{fig:GeneralTopRow}. Let $a_c$ and $b_c$ represent the entries given by a restricted nonattacking filling $\ws$. In this scenario, we assume $a_0 = b_0$ and all other entries are distinct, $a_c \neq b_{c'}$ where $c \neq 0$ and $c' \neq 0$. We define $\delta_c = 0$ if $\{a_0, a_c, b_c\}$ is not a type II coinversion triple in $\ws$, and $\delta_c = 1$, otherwise.

\begin{figure}[h!]
    \centering
    \begin{tikzpicture}[scale = 0.8]
        \draw[thick] (0,0) grid (3,2);
        \draw[thick] (6,0) grid (8,2);

        \node[] at (4.5,1) {$\ldots$};

        \node[] at (0.5,1.5) {$a_0$};
        \node[] at (1.5,1.5) {$a_1$};
        \node[] at (2.5,1.5) {$a_2$};
        \node[] at (6.5,1.5) {$a_{k-1}$};
        \node[] at (7.5,1.5) {$a_{k}$};

        \node[] at (0.5,0.5) {$b_0$};
        \node[] at (1.5,0.5) {$b_1$};
        \node[] at (2.5,0.5) {$b_2$};
        \node[] at (6.5,0.5) {$b_{k-1}$};
        \node[] at (7.5,0.5) {$b_{k}$};

        \node[] at (1.5,2.5) {$\delta_1$};
        \node[] at (2.5,2.5) {$\delta_2$};
        \node[] at (6.5,2.5) {$\delta_{k-1}$};
        \node[] at (7.5,2.5) {$\delta_{k}$};
    \end{tikzpicture}
    \caption{A representation of the top row of $k$ columns in the augmented diagram of a composition $\mu$.}
    \label{fig:GeneralTopRow}
\end{figure}
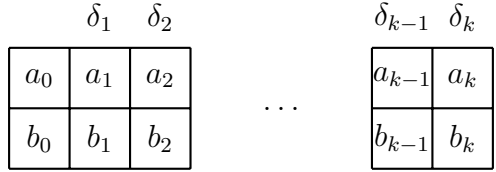

There are $(k-1)!$ restricted nonattacking fillings generated by permuting the entries $a_1, a_2, \ldots, a_k$. There are $k!$ nonattacking fillings that correspond to permuting all the entries $a_0, a_1, \ldots, a_k$. Our goal is to partition the nonattacking fillings into $(k-1)!$ classes such that each class contains a unique restricted nonattacking filling and the sum of the $\HHL$ $q,t$-weights of a class equals the mixed $q,t$-weight of the restricted nonattacking filling. 

We show an example in Figure~\ref{fig:k=4Example} for $k = 3$. In this case, moving $a_0=3$ to the right by transposing adjacent entries is not enough.  

\begin{figure}[h!]
    \centering
    \begin{tikzpicture}[scale = 0.5]
        \begin{scope}
            \draw[thick] (0,0) grid (4,2);

            \node[] at (0.5,1.5) {$3$};
            \node[] at (1.5,1.5) {$1$};
            \node[] at (2.5,1.5) {$2$};
            \node[] at (3.5,1.5) {$5$};

            \node[] at (0.5,0.5) {$3$};
            \node[] at (1.5,0.5) {$4$};
            \node[] at (2.5,0.5) {$6$};
            \node[] at (3.5,0.5) {$7$};

            \node[] at (1.5,2.5) {$1$};
            \node[] at (2.5,2.5) {$1$};
            \node[] at (3.5,2.5) {$0$};

            \node[] at (2, -1) {$\frac{t^2(1-t)^3}{(1-qt^2)(1-qt^3)(1-qt^4)}$};
        \end{scope}
        \begin{scope}[xshift=8cm]
            \draw[thick] (0,0) grid (4,2);

            \node[] at (0.5,1.5) {$1$};
            \node[] at (1.5,1.5) {$3$};
            \node[] at (2.5,1.5) {$2$};
            \node[] at (3.5,1.5) {$5$};

            \node[] at (0.5,0.5) {$3$};
            \node[] at (1.5,0.5) {$4$};
            \node[] at (2.5,0.5) {$6$};
            \node[] at (3.5,0.5) {$7$};

            \node[] at (2, -1) {$\frac{t(1-t)^4}{(1-qt)(1-qt^2)(1-qt^3)(1-qt^4)}$};
        \end{scope}
        \begin{scope}[xshift=16cm]
            \draw[thick] (0,0) grid (4,2);

            \node[] at (0.5,1.5) {$1$};
            \node[] at (1.5,1.5) {$2$};
            \node[] at (2.5,1.5) {$3$};
            \node[] at (3.5,1.5) {$5$};

            \node[] at (0.5,0.5) {$3$};
            \node[] at (1.5,0.5) {$4$};
            \node[] at (2.5,0.5) {$6$};
            \node[] at (3.5,0.5) {$7$};

            \node[] at (2, -1) {$\frac{(1-t)^4}{(1-qt)(1-qt^2)(1-qt^3)(1-qt^4)}$};
        \end{scope}
        \begin{scope}[xshift=24cm]
            \draw[thick] (0,0) grid (4,2);

            \node[] at (0.5,1.5) {$5$};
            \node[] at (1.5,1.5) {$2$};
            \node[] at (2.5,1.5) {$1$};
            \node[] at (3.5,1.5) {$3$};

            \node[] at (0.5,0.5) {$3$};
            \node[] at (1.5,0.5) {$4$};
            \node[] at (2.5,0.5) {$6$};
            \node[] at (3.5,0.5) {$7$};

            \node[] at (2, -1) {$\frac{qt^3(1-t)^4}{(1-qt)(1-qt^2)(1-qt^3)(1-qt^4)}$};
        \end{scope}
        \node[] at (15,-3.25) {$\sum \wtqt(\ws) = \dfrac{(1-t)^3}{(1-qt)(1-qt^2)(1-qt^3)}$};
    \end{tikzpicture}
    \caption{Example of a class whose $q,t$-weights simplify as to the form in Conjecture~\ref{conj:MainConj} }
    \label{fig:k=4Example}
\end{figure}

Even though finding the exact filling with the correct weight seems to be a very difficult question, it is not hard to determine the exact weight that this unknown filling needs to be in order for the simplification to occur as expected. 
We are currently working on a way to determine the classes in the general case and plan to explore Conjecture~\ref{conj:MainConj} for the permuted basement nonattacking fillings~\cite{A19}. These are essentially the same as nonattacking fillings; the only difference is that the basement is a specified permutation of the entries. Conjecture~\ref{conj:MainConj} extends naturally to these fillings since the definition of attacking only depends on the structure of the augmented diagram of a composition. 

\bibliographystyle{amsplain}
\bibliography{biblio_Alex}

\end{document}